\documentclass[12pt]{amsart}
\pdfoutput=1
\usepackage{blindtext}

\usepackage{multicol}
\usepackage[british]{babel} 
\usepackage[active]{srcltx}
\usepackage{subfigure} 
\usepackage{amssymb}
\usepackage{amsmath,amstext,amsthm,amsfonts,latexsym}
\usepackage[colorlinks=true]{hyperref} 
\hypersetup{linkcolor=blue,	citecolor=red,pdftitle={Uniform gap in Lyapunov exponents and dominated splitting for linear cocycles},pdfauthor={Bruno Yemini},}
\usepackage{lineno} 
\usepackage{verbatim}

\usepackage{soul} 
\usepackage{graphicx}

\makeatletter
\def\referencia#1#2{\begingroup
#2%
\def\@currentlabel{#2}%
\phantomsection\label{#1}\endgroup
}
\makeatother

\usepackage{amsmath}
\usepackage{geometry}
\usepackage{graphics}

\usepackage{amsthm}
\usepackage{amsfonts}
\usepackage{amsmath}
\usepackage{amssymb}
\usepackage{mathtools}

\usepackage[utf8]{inputenc}
\usepackage[T1]{fontenc}
\usepackage{microtype}
\usepackage{xcolor}

\usepackage[british]{babel}
\usepackage[colorinlistoftodos]{todonotes}

\theoremstyle{plain}
\newtheorem{teo}{\textsc{Theorem}}[section]
\newtheorem*{mnteo}{\textsc{Main Theorem}}
\newtheorem*{question}{\textsc{Question}}
\newtheorem{quesnum}{\textsc{Question}}
\newtheorem*{ques2rev}{\textsc{Question 2'}}
\newtheorem{prop}[teo]{Proposition}
\newtheorem{lema}[teo]{Lemma}

\newtheorem{coro}{\textbf{Corollary}}

\theoremstyle{definition}
\newtheorem{defi}[teo]{Definition}

\theoremstyle{remark}
\newtheorem{obs}[teo]{\textit{Remark}}

\newtheorem*{rem}    {Remark}

\newcommand{\N}     {\mathbb{N}}
\newcommand{\Z}     {\mathbb{Z}}

\newcommand{\R}     {\mathbb{R}}

\renewcommand{\P}   {\mathbb{P}}

\newcommand{\reg}   {\mathcal{R}}
\newcommand{\per}   {\mathcal{P}{\hspace{-2pt}}\text{\textit{er}}}

\newcommand{\e}     {\varepsilon}

\newcommand{\supp}  {\text{\textup{supp}}\, }
\renewcommand{\exp} {\text{\textup{exp}}}

\newcommand{\gld} {\text{\textup{GL}} (d,\mathbb{R})}
\newcommand{\sld} {\ensuremath{\text{\textup{SL}} (d,\mathbb{R})}}

\newcommand\restrict[1]{{|}_{#1}}

\newcommand{\inte}[1]{\ensuremath{\text{\textup{int}}({#1})}}
\newcommand{\Hom}[2][\empty]{\ensuremath{\text{\textup{Homeo}}_{#1} ({#2})}}
\newcommand{\Diff}[2][\empty]{\ensuremath{\text{\textup{Diff}}_{#1}^1({#2})}}

\title[Uniform Gap in Lyapunov Exponents]{Uniform gap in Lyapunov Exponents and Dominated Splitting for Linear Cocycles}

\author[Bruno Yemini]{}

\email{byemini@fing.edu.uy}
\thanks{The author was supported by PUCV Postgraduate Scholarship Program 2017 (Be\-ne\-fi\-cia\-rio Be\-ca 
Post\-gra\-do PUCV 2017).
The author was partially supported by Proyecto FONDECYT Regular 1171477 and Proyecto FONDECYT Regular 1210168.}



\usepackage{pgfplots}
\pgfplotsset{compat=1.15}
\usepackage{mathrsfs}
\usetikzlibrary{arrows}

\begin{document}

\maketitle

\centerline{\scshape Bruno Yemini}
\medskip
{\footnotesize
 \centerline{Instituto de Matemática y Estadística ``Rafael Laguardia'', Facultad
 de Ingeniería. Universidad de la República.}
 \centerline{Julio Herrera y Reissig 565, C.P. 11300}
 \centerline{Montevideo, Uruguay}
} 

\bigskip

\begin{abstract}
Given a linear cocycle over an ergodic homeomorphism on a compact metric space, we show that the existence of a uniform gap between the \(p\)-th and \((p+1)\)-th Lyapunov exponents on a \(C^0\)-neighbourhood implies the existence of a dominated splitting of index \(p\).
\end{abstract}

\tableofcontents

\section{Introduction}
The theory of characteristic exponents was founded at the end of the nineteenth century by Lyapunov~\cite{Lyapunov1892}, in his study of the stability of solutions of ordinary differential equations.
The works of Furstenberg and Kesten~\cite{FurstenbergKesten1960} and of Oseledets~\cite{Oseledets68} brought the concept into ergodic theory.
Lyapunov exponents appear naturally in smooth dynamical systems, after the work of Pesin~\cite{Pesin77}, and in spectral theory~\cite{Damanik2007}.

Let \(X\) be a compact metric space and \(f \in \Hom{X}\).
Let \(A:X \rightarrow \mathcal{G}\) be a continuous map, where \(\mathcal{G}\) is either \(\gld\) or \(\sld\).
The \emph{linear cocycle} defined by \(A\) is the map \(F_A:X\times \R^d \rightarrow X\times \R^d\) given by \((x,v) \mapsto (f(x),A(x)v)\).
By Oseledets' theorem, there is a set of full probability on which an \(f\)-invariant splitting \(\R^d = E_1(x) \oplus \cdots \oplus E_{k(x)}(x)\) is defined such that, for every non-zero \(v \in E_s(x)\),
\[
	\lim_{n \to \infty} \frac{1}{n} \log \|A^n(x)v\| = \gamma_s(x).
\]
This splitting is called the \emph{Oseledets splitting}, and the numbers \(\gamma_i(x)\) are the \emph{Lyapunov exponents}.
When \(\mu\) is ergodic, the Lyapunov exponents and the dimensions of the Oseledets subspaces are constant \(\mu\)-almost everywhere.
Since this is the case throughout the paper, we omit the point \(x\) from the notation for the Lyapunov exponents.
We denote the Lyapunov exponents counted with multiplicity by \(\lambda_1 \geq \lambda_2 \geq \cdots \geq \lambda_d\).
This is the \emph{Lyapunov spectrum} of the cocycle.

Every non-zero vector of \(E_s(x)\) grows exponentially, at the rate given by the corresponding Lyapunov exponent.
The Oseledets splitting depends measurably on the point, and it encodes the asymptotic behaviour of the cocycle.
Understanding how the Lyapunov exponents and the Oseledets splitting depend on the cocycle has become a broad and active field of study.
For a survey on the topic we refer the reader to~\cite{wilkinson17}.

More generally, let \(\mathcal{E}\) be a vector bundle over a topological space \(X\) and let \(F_A:\mathcal{E}\rightarrow \mathcal{E}\) be a linear cocycle induced by a map \(A:X\rightarrow \mathcal{G}\).
An invariant splitting \(\mathcal{E} = E \oplus F\) is called \emph{hyperbolic} if the vectors in \(E\) are uniformly expanded and the vectors in \(F\) are uniformly contracted by the action of \(A\) on the fibres.
Hyperbolicity is a very strong condition, and weaker forms of it have been introduced and widely studied; one of them is \emph{domination}.
Given \(m \in \N\), an invariant splitting \(E \oplus F\) is an \emph{\(m\)-dominated splitting} if, for every \(x\),
    \[
        \frac{\left\Vert A^m{(x)}\restrict{F(x)}\right\Vert}{ \text{\textup{\textbf{m}}}\left(A^m{(x)}\restrict{E(x)}\right) } \leq \frac{1}{2}
    \]
where \(\text{\textup{\textbf{m}}}(L) = \| L^{-1}\|^{-1}\) denotes the minimum norm of a linear map \(L\).
We say that \(E \oplus F\) is a \emph{dominated splitting} if it is \(m\)-dominated for some \(m \in \N\).

The notion of dominated splitting played an important role in the works of Ma\~n\'e~\cite{Mane78}, Pliss~\cite{Pliss72} and Liao~\cite{Liao80} on Smale's stability conjecture.
The term ``dominated splitting'' was introduced by Ma\~n\'e in~\cite{Mane84}.
The same concept was used independently in the theory of ordinary differential equations by the Russian school, under the name of \emph{exponential separation} (see~\cite{Palmer82} for further details).
Dominated splittings remain an important subject in dynamical systems~\cite{Pujals2007} and in control theory~\cite{ColKli00}.
Although it was originally defined for smooth dynamical systems, the notion of dominated splitting extends naturally to linear cocycles.

It is natural to ask when the Oseledets splitting is a dominated splitting.
This is a strong property: since dominated splittings are continuous, it means that the Oseledets splitting, which a priori is only measurable and defined almost everywhere, extends continuously to the whole support of the measure.
In this paper we address this question one gap of the Lyapunov spectrum at a time.

Fix \(p \in \{1,\ldots,d-1\}\) and let \(A\) be a cocycle such that \(\lambda_p(A) > \lambda_{p+1}(A)\).
Let \(\gamma_1 > \cdots > \gamma_{k}\) be the Lyapunov exponents of \(A\) counted without multiplicity, let \(s \in \{1,\ldots,k-1\}\) be given by \(\gamma_s = \lambda_p(A)\) and \(\gamma_{s+1} = \lambda_{p+1}(A)\), and group the Oseledets subspaces as
\begin{equation}\label{split}
E = \bigoplus_{i=1}^s E_i, \quad \quad F= \bigoplus_{i=s+1}^{k} E_i.
\end{equation}
Then \(E \oplus F\) is a measurable invariant splitting, defined \(\mu\)-almost everywhere, with \(\dim E = p\), which separates the \(p\) largest Lyapunov exponents from the remaining ones.
We call it the \emph{Oseledets splitting of index \(p\)} of \(A\).

If \(A\) admits a dominated splitting of index \(p\) over \(\supp(\mu)\), then \(\lambda_p(A) > \lambda_{p+1}(A)\) and the dominated splitting coincides almost everywhere with the Oseledets splitting of index \(p\).
The converse is false: a gap between two consecutive Lyapunov exponents does not force \(E \oplus F\) to be dominated.
This is the case, for instance, of quasi-periodic Schr\"odinger cocycles at energies in the spectrum for which the Lyapunov exponent is positive; see~\cite{BochiViana_2005,Damanik2007}.
Domination is a uniform property, so some uniformity of the gap has to be required.

\begin{question}
Assume that \(\lambda_p(A) > \lambda_{p+1}(A)\).
What kind of uniformity of this gap forces the Oseledets splitting of index \(p\) to be dominated?
\end{question}

In the \(C^0\) topology, the fundamental result in this direction is due to Bochi and Viana~\cite{BochiViana_2005}, and it is formulated in terms of \emph{continuity}: if \(A\) is a point of continuity of the Lyapunov spectrum, regarded as a function of the cocycle, then the Oseledets splitting of \(A\) is either trivial or dominated.
Since the Lyapunov spectrum is continuous at \(C^0\)-generic cocycles, this dichotomy holds on a residual subset of the space of linear cocycles.
The mechanism behind this result goes back to Ma\~n\'e~\cite{Mane84}: without domination, small perturbations of the cocycle can mix the directions of \(E\) and \(F\), and this makes the sum of the \(p\) largest Lyapunov exponents drop.

In this paper we show that the same mechanism answers the question above when continuity is replaced by \emph{robustness}: a gap between \(\lambda_p\) and \(\lambda_{p+1}\) that persists, with a uniform size, under \(C^0\)-small perturbations of the cocycle can only be produced by a dominated splitting of index \(p\).

\subsection*{Main result}
We say that the Lyapunov spectrum \(\lambda_1(A) \geq \cdots \geq \lambda_d(A)\) of a cocycle \(A \in \text{\textup{C}}(X,\mathcal{G})\) has the \emph{uniform \(p\)-gap property} if there exist a \(C^0\)-neighbourhood \(\mathcal{U}_A\) of \(A\) and \(\beta > 0\) such that
\[
\lambda_{p}(B) - \lambda_{p+1}(B) \geq \beta \quad \text{for every } B \in \mathcal{U}_A.
\]
The number \(\beta\) is called a \emph{\(p\)-gap} of \(A\).

\begin{mnteo}
Let \(X\) be a compact metric space, let \(\mu\) be a Borel probability measure on \(X\) and let \(f \in \Hom[\mu]{X}\) be such that \(\mu\) is ergodic.
Let \(A \in \text{\textup{C}}(X,\mathcal{G})\) be a linear cocycle over \(f\) whose Lyapunov spectrum has the uniform \(p\)-gap property.
Then \(A\) admits a dominated splitting of index \(p\) over \(\supp(\mu)\).
\end{mnteo}

This dominated splitting coincides almost everywhere with the Oseledets splitting of index \(p\) defined in~\eqref{split}.
In other words, \(E \oplus F\) extends to a continuous splitting over \(\supp(\mu)\), and this extension is dominated.

Two reductions are used throughout the paper.
First, if \(\supp(\mu)\) is finite then \(\mu\) is supported on a periodic orbit and the conclusion of the Main Theorem is elementary: along a periodic orbit the splitting~\eqref{split} is given by sums of generalised eigenspaces of the return map, and it is dominated as soon as \(\lambda_p(A) > \lambda_{p+1}(A)\).
Hence we assume that \(\supp(\mu)\) is infinite.
Second, there is no loss of generality in assuming that \(\mu\) is fully supported.
Indeed, \(\supp(\mu)\) is a compact \(f\)-invariant set and, by Tietze's extension theorem, every continuous cocycle over \(f\restrict{\supp(\mu)}\) which is \(C^0\)-close to the restriction of \(A\) is the restriction of a cocycle in \(\text{\textup{C}}(X,\mathcal{G})\) which is \(C^0\)-close to \(A\).
(One extends the difference between the two cocycles coordinatewise, composes the extension with the radial retraction onto a small ball of the space of matrices and, when \(\mathcal{G} = \sld\), normalises the determinant.)
Since Lyapunov exponents only depend on the values of the cocycle on \(\supp(\mu)\), the uniform \(p\)-gap property is inherited by the restriction of \(A\) to \(\supp(\mu)\).

The converse of the Main Theorem is also true, and it is elementary.
Indeed, assume that \(A\) admits an \(m\)-dominated splitting of index \(p\) over \(\supp(\mu)\).
By the cone field criterion (see~\cite[Appendix~B.1]{BonattiDiazViana}), there exist a \(C^0\)-neighbourhood \(\mathcal{U}\) of \(A\) and \(m' \in \N\) such that every \(B \in \mathcal{U}\) admits an \(m'\)-dominated splitting \(E_B \oplus F_B\) of index \(p\) over \(\supp(\mu)\).
Iterating the domination inequality we get \(\|B^{km'}(x)w\| \leq 2^{-k}\|B^{km'}(x)v\|\) for every \(x \in \supp(\mu)\), every \(k \geq 1\) and all unit vectors \(v \in E_B(x)\) and \(w \in F_B(x)\).
Hence every Lyapunov exponent of \(B\) along \(E_B\) exceeds every Lyapunov exponent of \(B\) along \(F_B\) by at least \((\log 2)/m'\) and, since \(\dim E_B = p\), we conclude that \(\lambda_p(B) - \lambda_{p+1}(B) \geq (\log 2)/m'\) for every \(B \in \mathcal{U}\).
Therefore the Main Theorem yields the following characterisation: \emph{a cocycle admits a dominated splitting of index \(p\) over \(\supp(\mu)\) if and only if its Lyapunov spectrum has the uniform \(p\)-gap property.}

\subsection*{Relation with the theorem of Bochi and Viana}
We say that a cocycle \(A \in \text{\textup{C}}(X,\mathcal{G})\) is a \emph{point of continuity of the \(i\)-th Lyapunov exponent} if the function \(B \mapsto \lambda_i(B)\), defined on \(\text{\textup{C}}(X,\mathcal{G})\) with the \(C^0\) topology, is continuous at \(B = A\).
We say that \(A\) is a \emph{point of continuity of the Lyapunov spectrum} if this holds for every \(i \in \{1,\ldots,d\}\).
If \(A\) is a point of continuity of the \(p\)-th and the \((p+1)\)-th Lyapunov exponents and \(\lambda_p(A) > \lambda_{p+1}(A)\), then
\[
\lambda_p(B) - \lambda_{p+1}(B) \geq \tfrac{1}{2}\left(\lambda_p(A) - \lambda_{p+1}(A)\right)
\]
for every \(B\) in a \(C^0\)-neighbourhood of \(A\), so the Lyapunov spectrum of \(A\) has the uniform \(p\)-gap property.
Hence the Main Theorem has the following consequence.

\begin{coro}\label{corocontinuidad}
Let \(A \in \text{\textup{C}}(X,\mathcal{G})\) be a point of continuity of the \(p\)-th and the \((p+1)\)-th Lyapunov exponents such that \(\lambda_p(A) > \lambda_{p+1}(A)\).
Then \(A\) admits a dominated splitting of index \(p\) over \(\supp(\mu)\).
\end{coro}

Applying Corollary~\ref{corocontinuidad} at every index \(p\) such that \(\lambda_p(A) > \lambda_{p+1}(A)\) we recover, for cocycles with values in \(\gld\) or \(\sld\) over an ergodic base, the theorem of Bochi and Viana mentioned above.

\begin{coro}[Bochi--Viana~\cite{BochiViana_2005}]\label{corobochiviana}
Let \(A \in \text{\textup{C}}(X,\mathcal{G})\) be a point of continuity of the Lyapunov spectrum.
Then either all the Lyapunov exponents of \(A\) are equal, or the Oseledets splitting of \(A\) coincides almost everywhere with a dominated splitting over \(\supp(\mu)\).
\end{coro}

In~\cite{BochiViana_2005} this result is first proved for volume preserving diffeomorphisms, where the realisation of the perturbations relies on tools which are specific to the Lebesgue measure, such as Vitali's covering lemma, and it is then adapted to linear cocycles with values in a large class of matrix groups, over measures which are not necessarily ergodic.
Our setting is more restrictive in these two respects.
On the other hand, the proof presented here is written directly for continuous cocycles and uses only the topology of the base space \(X\), through Urysohn's lemma; in particular it applies to any compact metric space, including totally disconnected ones such as shift spaces.

The Main Theorem contains more information than Corollary~\ref{corobochiviana}, because it does not involve the continuity of the whole Lyapunov spectrum.
The following statement illustrates this point.

\begin{coro}\label{mncoro}
Let \(A \in \text{\textup{C}}(X,\text{\textup{SL}}(3,\mathbb{R}))\) be a point of continuity of the second Lyapunov exponent.
Then either \(\lambda_2(A) = 0\) or the whole Oseledets splitting of \(A\) is dominated.
\end{coro}

The proof of Corollary~\ref{mncoro}, given in Section~5, uses the full strength of the Main Theorem.
Suppose, for instance, that \(\lambda_2(A) < 0\), so that \(\lambda_1(A) > 0 > \lambda_2(A)\).
If \(A\) had no dominated splitting of index \(1\), the Main Theorem would provide cocycles \(B\) arbitrarily close to \(A\) with \(\lambda_1(B) - \lambda_2(B)\) arbitrarily small; since \(\lambda_2(B)\) is close to \(\lambda_2(A) < 0\), this is incompatible with \(\lambda_1(B) + \lambda_2(B) + \lambda_3(B) = 0\).
In contrast, one application of the perturbation of~\cite{BochiViana_2005} to the cocycle \(A\) is only guaranteed to reduce \(\lambda_1\) by half of the gap \(\lambda_1(A) - \lambda_2(A)\), which is not enough to reach a contradiction in general.

\subsection*{Related results}
The Main Theorem belongs to a family of results that deduce domination from a gap in spectral data.
These results differ in the kind of uniformity that is required from the gap.

\emph{Uniformity in the base point.}
Bochi and Gourmelon~\cite{BochiGourmelon2009} proved that a cocycle admits a dominated splitting of index \(p\) if and only if there is a uniform exponential gap between the \(p\)-th and the \((p+1)\)-th singular values of its iterates: there exist \(C > 0\) and \(\tau \in (0,1)\) such that \(\sigma_{p+1}(A^n(x)) \leq C \tau^n \sigma_{p}(A^n(x))\) for every \(x \in X\) and every \(n \in \N\).
This hypothesis involves finite-time data of a single cocycle, and it is required at every point.

\emph{Uniformity in the invariant measure.}
Lyapunov exponents are asymptotic data, defined almost everywhere with respect to each invariant measure.
A dominated splitting of index \(p\) forces a gap between the \(p\)-th and the \((p+1)\)-th Lyapunov exponents which is uniform over all invariant measures, and the converse is known in the presence of additional structure.
Kassel and Potrie~\cite{kaspot2021} proved it for locally constant cocycles over subshifts of finite type, in connection with Anosov representations (see also~\cite{BochiPotrieSambarino2019}).
For H\"older cocycles over hyperbolic systems it holds under fibre bunching conditions, by Velozo Ruiz~\cite{Velozo2020} in dimension two and by Backes~\cite{Backes2026} in any dimension, and DeWitt and Gogolev~\cite{DeWittGogolev2024} obtained dominated splittings from constant, or narrow, periodic data; see also~\cite{Backes2025}.
When a continuous invariant splitting is given in advance, a uniform gap of Lyapunov exponents on a total probability set is equivalent to domination; see Salgado~\cite{Salgado2019} and also~\cite{AraujoSalgado2013,AraujoSalgado2015}.
Without additional structure the converse fails, even over a subshift of finite type: there are H\"older cocycles with a uniform gap~\cite{Velozo2020}, and continuous cocycles with the same non-zero Lyapunov exponents for all invariant measures~\cite{Bochi2024}, which are not uniformly hyperbolic.
Examples of this kind in smooth dynamics were given in~\cite{CaoLuzzattoRios2006}.

\emph{Uniformity in the cocycle.}
In the Main Theorem the gap is required for a single ergodic measure, the cocycle is merely continuous and no assumption is made on the base dynamics; the uniformity is with respect to \(C^0\)-perturbations of the cocycle.
As explained above, this hypothesis is also necessary.
In particular, for the cocycles of~\cite{Velozo2020,Bochi2024} the gap, although uniform over all invariant measures, is not robust: with respect to any fully supported ergodic measure, it can be made arbitrarily small by \(C^0\)-small perturbations of the cocycle.
For \(\text{\textup{SL}}(2,\mathbb{R})\)-cocycles the Main Theorem asserts that a robustly positive Lyapunov exponent implies uniform hyperbolicity, a statement which is close in spirit to, and essentially contained in the arguments of,~\cite{BOCHI_2002a}.

\emph{Periodic orbits.}
The perturbative point of view is classical for periodic orbits of diffeomorphisms in the \(C^1\) topology.
Bonatti, Gourmelon and Vivier~\cite{BonattiGourmelonVivier2006} proved that a periodic orbit of large period either admits a dominated splitting of a prescribed strength or can be perturbed so that all the eigenvalues of the return map have the same modulus, and Bochi and Bonatti~\cite{BochiBonatti2012} described all the Lyapunov spectra that can be obtained by perturbing the derivative along periodic orbits, in terms of the finest dominated splitting.
The Main Theorem can be seen as a counterpart of this principle for an ergodic measure of a continuous cocycle; see also the first question of Section~6.

\subsection*{Outline of the proof}
The whole argument follows one quantity.
For a cocycle \(B\) and \(1 \leq i \leq d\) we write
\[
\Lambda_i(B) = \lambda_1(B) + \cdots + \lambda_i(B),
\]
and we set \(\Lambda_0(B) = 0\).
The number \(\Lambda_p(B)\) is the largest Lyapunov exponent of the cocycle \(\wedge^p B\) induced by \(B\) on the \(p\)-th exterior power of \(\R^d\), hence
\[
\Lambda_p(B) = \inf_{n \geq 1} \frac{1}{n} \int_X \log \| \wedge^p B^n(x) \| \, d\mu(x).
\]
In particular, the function \(B \mapsto \Lambda_p(B)\) is upper semicontinuous in the \(C^0\) topology: under small perturbations of the cocycle, \(\Lambda_p\) can only jump down.
We show that, in the absence of domination, it does jump down, by a definite amount and in a uniform way.

Let \(A\) be a cocycle with \(\|A\| < M\), \(\|A^{-1}\| < M\) and \(\lambda_p(A) > \lambda_{p+1}(A)\), let \(E \oplus F\) be its Oseledets splitting of index \(p\), and fix \(\e > 0\) and \(\delta > 0\).

\smallskip
\noindent\emph{Step 1: exchanging directions along an orbit segment.}
We say that \(E \oplus F\) is not \(m\)-dominated at a point \(y\) if
\begin{equation}\label{outline1}
\| A^m(y)\restrict{F(y)} \| \, \|{(A^m(y)\restrict{E(y)})}^{-1}\| > \frac{1}{2}.
\end{equation}
There is \(m_0\), \emph{depending only on \(M\) and \(\e\)}, with the following property.
If~\eqref{outline1} holds for some \(m \geq m_0\) then, composing the matrices \(A(y), \ldots, A(f^{m-1}(y))\) with rotations and hyperbolic matrices close to the identity, we obtain matrices \(L_0, \ldots, L_{m-1}\) such that \(\|L_j - A(f^j(y))\| < \e\) for every \(j\) and such that the product \(L_{m-1} \cdots L_0\) sends some non-zero vector of \(E(y)\) into \(F(f^m(y))\) (Proposition~\ref{lemaperturbativo}).

\smallskip
\noindent\emph{Step 2: losing half of the gap along a long orbit segment.}
Suppose that~\eqref{outline1} holds on a set of positive measure, for some \(m \geq m_0\).
By ergodicity, for \(\mu\)-almost every \(x\) and every sufficiently large \(n\) there is \(\ell \approx n/2\) such that \(y = f^{\ell}(x)\) satisfies~\eqref{outline1}.
We replace \(A\) by the matrices of Step~1 along the orbit segment from \(y\) to \(f^m(y)\), and we keep \(A\) along the rest of the orbit segment of length \(n\) starting at \(x\).
The \(p\)-vector that spans \(\wedge^p E\) grows at rate \(\Lambda_p(A)\) during the first half of the segment; it is then sent into the sum of the remaining Oseledets subspaces of \(\wedge^p A\), where the growth rate during the second half is at most \(\Lambda_{p-1}(A) + \lambda_{p+1}(A)\).
As a consequence, the norm of the \(p\)-th exterior power of the product of the new sequence of \(n\) matrices is at most \(\exp\left(n\left(\tfrac{1}{2}(\Lambda_{p-1}(A) + \Lambda_{p+1}(A)) + \delta\right)\right)\) (Lemma~\ref{estimacionnorma}).
Note that
\[
\frac{\Lambda_{p-1}(A) + \Lambda_{p+1}(A)}{2} = \Lambda_p(A) - \frac{\lambda_p(A) - \lambda_{p+1}(A)}{2},
\]
that is, half of the gap is lost.

\smallskip
\noindent\emph{Step 3: realisation.}
Since \(\mu\)-almost every point is non-periodic, the first \(n\) iterates of a small neighbourhood of \(x\) are pairwise disjoint, and Urysohn's lemma allows us to interpolate between \(A\) and the matrices of Step~2 (for \(\mathcal{G} = \sld\) the interpolation is normalised by the determinant; keeping the interpolated matrices in \(\mathcal{G}\) and close to \(A\) is what requires the bound on \(\|A^{-1}\|\)).
We obtain a continuous cocycle, \(\e\)-close to \(A\), which coincides with \(A\) outside the iterates of that neighbourhood and which follows the sequence of Step~2 on the iterates of a prescribed compact subset of it (Proposition~\ref{mainperturbativo}).

\smallskip
\noindent\emph{Step 4: globalisation.}
Using a Kakutani--Rokhlin castle whose towers have open bases with boundaries of zero measure, we cover most of the space by pairwise disjoint families of iterates of such neighbourhoods, and we perform the perturbation of Step~3 on each of them.
The outcome is a cocycle \(B\) such that \(\|A - B\| < \e\) and
\begin{equation}\label{outline2}
\Lambda_p(B) < \Lambda_p(A) - \frac{\lambda_p(A) - \lambda_{p+1}(A)}{2} + \delta
\end{equation}
(Theorem~\ref{bajaexponente}).

\smallskip
\noindent\emph{Step 5: an extremal argument.}
For a fixed cocycle, Steps~1 to~4 are those of~\cite{BochiViana_2005}, and they show that \(\Lambda_p\) is discontinuous at every cocycle whose Oseledets splitting of index \(p\) is not dominated.
The key point for us is that the threshold \(m_0\) of Step~1 depends on the cocycle only through the bound \(M\).
Assume that \(A\) has the uniform \(p\)-gap property, with \(p\)-gap \(\beta\) on a neighbourhood \(\mathcal{U}_A\) whose elements \(B\) all satisfy \(\|B\| < M\) and \(\|B^{-1}\| < M\), and suppose that \(A\) admits no dominated splitting of index \(p\).
Let
\[
\alpha = \lim_{r \rightarrow 0} \, \inf \left\{ \Lambda_p(B) \, : \, \|B - A\| < r \right\},
\]
which is a finite number, and fix \(\e > 0\) such that the ball of radius \(\e\) centred at \(A\) is contained in \(\mathcal{U}_A\) and \(\Lambda_p(B) > \alpha - \beta/8\) for every \(B\) in this ball.
Let \(m_0\) be a threshold for which Steps~1 to~4 work, with perturbations of size \(\e/2\), for every cocycle \(A'\) with \(\|A'\| < M\) and \(\|{A'}^{-1}\| < M\).
The set of cocycles admitting an \(m_0\)-dominated splitting of index \(p\) is closed in the \(C^0\) topology; therefore no cocycle sufficiently close to \(A\) admits such a splitting.
Since \(\mu\) is ergodic and fully supported, this means that for each of these cocycles the inequality~\eqref{outline1}, with \(m = m_0\) and with respect to its own Oseledets splitting of index \(p\), holds on a set of positive measure.
Among these cocycles we choose \(A'\) with \(\|A' - A\| < \e/2\) and \(\Lambda_p(A') < \alpha + \beta/8\), that is, an almost minimiser of \(\Lambda_p\) near \(A\).
Applying~\eqref{outline2} to \(A'\), with \(\delta = \beta/8\), we obtain a cocycle \(B\) such that \(\|B - A'\| < \e/2\) and
\[
\Lambda_p(B) < \Lambda_p(A') - \frac{\beta}{2} + \frac{\beta}{8} < \alpha - \frac{\beta}{4}.
\]
Since \(\|B - A\| < \e\), this contradicts the choice of \(\e\).

\smallskip
The general scheme of Steps~1 to~4 is due to Bochi and Viana~\cite{BochiViana_2005}; in particular, Lemma~\ref{estimacionnorma} is a version of~\cite[Proposition~4.2]{BochiViana_2005}.
The contributions of this paper are the following.
First, in Step~1 we give a direct construction of the sequence of matrices, with a threshold \(m_0\) that depends only on \(M\) and \(\e\).
This is possible because, for continuous cocycles, every sequence of matrices close to the cocycle along an orbit segment can be realised by a perturbation; hence there is no need for the nested rotations that are used in~\cite{BochiViana_2005} to deal with volume preserving diffeomorphisms, and one can use hyperbolic matrices instead.
Second, Steps~3 and~4 use only the topology of the base space: Urysohn's lemma and towers whose bases have boundaries of zero measure replace the covering arguments of~\cite{BochiViana_2005}.
Third, and most importantly, Step~5 is new: the perturbation is applied to an almost minimiser of \(\Lambda_p\) instead of the cocycle \(A\) itself, and this turns the semicontinuity argument of~\cite{BochiViana_2005} into a statement in which no continuity is assumed.
We remark that the uniformity of the threshold \(m_0\) is essential here.
Indeed, the results of~\cite{BochiViana_2005}, applied to cocycles near \(A\), by themselves only show that the almost minimisers of \(\Lambda_p\) near \(A\) admit dominated splittings of index \(p\), and the strength of the domination could degenerate as these cocycles approach \(A\).

\subsection*{Organisation of the paper}
In Section~\ref{definitions} we give the basic definitions and collect the properties of Lyapunov exponents and dominated splittings that are used in the paper.
In Section~3 we introduce space exchanging sequences, we prove their existence in the absence of domination (Step~1) and we estimate the norm of the corresponding products (Step~2).
In Section~4 we realise these sequences as perturbations of the cocycle and we globalise the construction (Steps~3 and~4).
In Section~5 we prove the Main Theorem (Step~5) and Corollary~\ref{mncoro}.
In Section~6 we pose some questions.


\section{Basic definitions and preliminaries}\label{definitions}

In this section we give the definitions and fix the notation used throughout the paper.
Let \(X\) be a compact metric space endowed with a Borel probability measure \(\mu\).
Recall that every Borel probability measure on a metric space is \emph{regular} (see~\cite[Theorem~1.1]{Billingsley1999}): for every Borel set \(S \subset X\),
\begin{align*}
\mu(S) &= \inf \{\mu(U)\ :\ S \subset U,\ U \text{ is an open set}\} \\
       &= \sup \{\mu(K)\ :\ K \subset S,\ K \text{ is a closed set}\}.
\end{align*}
Since \(X\) is compact, every closed set is compact, so the supremum above is taken over compact sets; that is, \(\mu\) is \emph{tight}.

As explained in the Introduction, we assume without loss of generality that \(\supp(\mu)\) is infinite and that \(\mu\) is fully supported.
The latter is equivalent to \(\mu\) being \emph{strictly positive}, that is, \(\mu(U) > 0\) for every non-empty open set \(U \subset X\).
In particular, \(X\) is an infinite compact metric space.

These hypotheses on the measure are satisfied in many important examples, such as the measure induced by a volume form on a compact Riemannian manifold, or a Bernoulli measure on the shift space \({\{1,\ldots,N\}}^{\Z}\).

A map \(f:X \rightarrow X\) \emph{preserves the measure \(\mu\)} if \(\mu(f^{-1}(S)) = \mu(S)\) for every measurable set \(S \subset X\).
We denote by \(\Hom[\mu]{X}\) the set of homeomorphisms that preserve the measure \(\mu\).
The measure \(\mu\) is \emph{ergodic} if \(\mu(S) \in \{0,1\}\) for every measurable set \(S\) such that \(f(S) = S\).
Since \(\mu\) is ergodic and fully supported and \(X\) is infinite, \(\mu\) is \emph{aperiodic}: the set \(\per(f)\) of periodic points of \(f\) satisfies \(\mu(\per(f))=0\).

\subsection{Linear cocycles}
Let \(f \in \Hom[\mu]{X}\).
Let \(\mathcal{G}\) denote either of the matrix groups \(\gld\) or \(\sld\).
To a continuous map \(A: X \rightarrow \mathcal{G}\) we associate the \emph{linear cocycle} on \(X \times \R^d\)
\[
    F_A:X\times \R^d \rightarrow X \times \R^d, \quad F_A(x,v)=(f(x),A(x)v).
\]
Since the base \(f\) is fixed, by an abuse of notation we identify the linear cocycle \(F_A\) with the map \(A: X \rightarrow \mathcal{G}\).
We denote by \(\text{\textup{C}}(X,\mathcal{G})\) the space of continuous maps from \(X\) to \(\mathcal{G}\), that is, the space of linear cocycles over \(f\).

We write the iterates of the cocycle as
\[
F_A^n(x,v) = (f^n(x),A^n(x)v).
\]
The maps \(A^n\) satisfy the \emph{cocycle property} \(A^{n+m}(x) =
A^n(f^m(x))A^m(x)\) for every \(n,m \in \Z\).

For \(L \in \mathcal{G}\) we use the \emph{operator norm} \[\|L\| := \max\{\|Lv\|:\ v \in \R^d,\ \|v\| = 1\}.
\]
This induces the \emph{\(C^0\) or uniform norm} on \(\text{\textup{C}}(X,\mathcal{G})\) as follows: given \(A \in \text{\textup{C}}(X,\mathcal{G})\),
\[
\|A\| = \max_{x \in X}\{\|A(x)\|\}.
\]
This norm induces a distance on \(\text{\textup{C}}(X,\mathcal{G})\).
To distinguish them from open subsets of \(X\), we denote open subsets of \(\text{\textup{C}}(X,\mathcal{G})\) by calligraphic letters, such as \(\mathcal{U}\) and \(\mathcal{V}\).
In particular, we denote the open balls of centre \(A\) and radius \(r\) in this space by \(\mathcal{U}(A,r)\).

The continuity of \(A\) implies that \(\log^+ \| A^{\pm 1}\| \in L^1(\mu)\).
 We say that \(\gamma\) is a
\emph{Lyapunov exponent} for the cocycle at the point \(x\) if there exists \(v \in \R^d\)
different from zero such that
\[
    \lim_n \frac{1}{n} \log \Vert A^n{(x)}v\Vert = \gamma,
\]
in this case we say that \(v\) is a \emph{Lyapunov vector} for \(x\).

A point \(x \in X\) is \emph{regular} if every vector \(v \in \R^d\) different from
zero is a Lyapunov vector.
We denote by \(\reg(A)\) the set of regular points of the cocycle \(A\).
Oseledets' theorem guarantees that the
set \(\reg(A)\) has \emph{full probability}, that is, it has full measure
for any \(f\)-invariant probability.

Oseledets' theorem also gives, for \(x \in \reg(A)\), an invariant splitting \(\R^d =E_1(x)\oplus E_2(x) \oplus \cdots \oplus E_{k(x)}(x)\) such that, for every non-zero \(v \in E_i(x)\),
\[
    \lim_n \frac{1}{n}\log \left\Vert A^n(x)v \right\Vert = \gamma_i(x).
\]
The splitting \(\R^d = E_1(x)\oplus\cdots \oplus E_{k(x)}(x)\) is called the \emph{Oseledets splitting}.

We define the \emph{multiplicity} of the exponent \(\gamma_i(x)\) as the
dimension of
\(E_i(x)\) and we denote the Lyapunov exponents, counted with multiplicity,
by \(\lambda_1(x) \geq \lambda_2(x)\geq \cdots \geq
\lambda_d(x)\).
 This is the \emph{Lyapunov spectrum} of the cocycle.

When \(\mu\) is ergodic, the invariant functions \(\lambda_i(x)\) and \(k(x)\) are constant almost everywhere.
Hence we omit the dependence on \(x\) of the Lyapunov exponents and of the number of subspaces of the Oseledets splitting, and we regard them as functions of the cocycle \(A\).

We now recall some facts about Lyapunov exponents of exterior powers; see~\cite{Viana2014,Arnold98} for further details.

Let \(\ell \in \{1,\ldots,d\}\). Every linear map \(L \in \mathcal{G}\) induces a linear map \(\wedge^{\ell}L\) on the exterior power \(\wedge^{\ell}\R^d\), which maps every decomposable vector \(v_1 \wedge \cdots \wedge v_{\ell}\) to \(L(v_1)\wedge \cdots \wedge L(v_{\ell})\).
Hence every cocycle \(A \in \text{\textup{C}}(X,\mathcal{G})\) induces a cocycle \(\wedge^{\ell} A\) on \(\wedge^{\ell}\R^d\).

The exterior power \(\wedge^{\ell}\R^d\) is endowed with the inner product for which the norm \(\| v_1 \wedge \cdots \wedge v_{\ell} \|\) of a decomposable vector is the volume of the \(\ell\)-dimensional parallelepiped spanned by \(v_1, \ldots, v_{\ell}\).
 The operator norm for linear maps and the \(C^0\) norm are defined analogously to their \(\R^d\) counterparts.

If \(A\in \text{\textup{C}}(X,\mathcal{G})\) has Lyapunov spectrum \(\lambda_1(A) \geq \cdots \geq \lambda_d(A)\), then the Lyapunov exponents of \(\wedge^{\ell} A\), counted with multiplicity, are the sums \(\lambda_{i_1}(A) +\cdots +\lambda_{i_{\ell}}(A)\) with \(1 \leq i_1 < \cdots < i_{\ell} \leq d\).
In particular, the two largest Lyapunov exponents of \(\wedge^{\ell} A\) are \(\lambda^{\wedge \ell}_1(A) = \lambda_1(A) + \cdots + \lambda_{\ell}(A)\) and \(\lambda^{\wedge \ell}_2(A) = \lambda_1(A) + \cdots + \lambda_{\ell - 1} (A) + \lambda_{\ell + 1}(A)\).

\subsection{Dominated splittings}

Let \(\Lambda \subset X\) be an \(f\)-invariant set, and suppose that for every \(x
\in \Lambda\) there is a splitting \(\R^d = E(x) \oplus F(x)\) into non-zero subspaces of constant dimensions;
this defines a splitting \(T_{\Lambda} = E \oplus F\) of the trivial bundle \(\Lambda \times \R^d\).
 The \emph{index} of the splitting is \(\dim(E(x))\).

We say that the splitting is \emph{invariant} under the cocycle \(A\) if
\(A(x)E(x)=E(f(x))\) and \(A(x)F(x)=F(f(x))\) for every \(x \in \Lambda\).

\begin{defi}
    Given \(m \in \N\), we say that \(T_{\Lambda} = E \oplus F\) is an
\emph{\(m\)-dominated splitting} if it is invariant and, for every \(x \in \Lambda\),
    \[
        \frac{\left\Vert A^m{(x)}\restrict{F(x)}\right\Vert}{ \text{\textup{\textbf{m}}}\left(
       A^m{(x)}\restrict{E(x)}\right) } \leq \frac{1}{2}
    \]
where \(\text{\textup{\textbf{m}}}(L) = \| L^{-1}\|^{-1}\) denotes the minimum norm of a linear map \(L\).

    We say that \(T_{\Lambda} = E \oplus F\) is a \emph{dominated splitting} if
    it is \(m\)-dominated for some \(m \in \N\).
In this case we write \(E \succ F\).
\end{defi}

Moreover, if we have a splitting $T_{\Lambda} = E_1 \oplus E_2 \oplus
E_3$, we say it is a \emph{dominated splitting} if $E_1 \succ E_2 \succ
E_3$ or, equivalently, $E_1 \succ (E_2 \oplus E_3)$ and $(E_1 \oplus E_2) \succ
E_3$.
This generalises in the obvious way to splittings into any number (at most
$d$) of subspaces.

We list some basic properties of dominated splittings; for details and proofs see~\cite{BonattiDiazViana}.

\begin{itemize}
\item \emph{Transversality.} If \(T_{\Lambda} = E \oplus F\) is a dominated splitting, then the angle \(\measuredangle(E(x),F(x))\) is bounded away from zero for all \(x \in \Lambda\).

\item \emph{Uniqueness.} If \(T_{\Lambda} = E^1 \oplus F^1\) and \(T_{\Lambda} = E^2 \oplus F^2\) are two dominated splittings of the same index, then \(E^1 = E^2\) and \(F^1 = F^2\).
\item \emph{Continuity.} A dominated splitting \(T_{\Lambda} = E\oplus F\) is continuous and extends continuously to a dominated splitting over the closure of \(\Lambda\).

\item \emph{Persistence under perturbations.} Let \(\Lambda_A\) be a compact invariant set over which \(A\) admits a dominated splitting. Then there exist a neighbourhood \(U\) of \(\Lambda_A\) and a \(C^0\)-neighbourhood \(\mathcal{U}\) of \(A\) in \(\text{\textup{C}}(X,\mathcal{G})\) such that, for every \(B \in \mathcal{U}\) and every compact \(f\)-invariant set \(\Lambda_B \subset U\), the cocycle \(B\) admits a dominated splitting over \(\Lambda_B\).

\end{itemize}

We say that an invariant splitting \(E \oplus F\) is \emph{not \(m\)-dominated at} \(y \in X\) if
\begin{equation}\label{nondom}
\frac{\| A^m(y)\restrict{F(y)} \|}{\text{\textup{\textbf{m}}}(A^m(y)\restrict{E(y)})} > \frac{1}{2}.
\end{equation}
We say that \(E \oplus F\) is \emph{eventually not \(m\)-dominated} at \(x \in X\) if there exists \(n \in \Z\) such that the splitting is not \(m\)-dominated at \(f^n(x)\).

Given \(m \in \N\), let \(\mathcal{D}_{E\oplus F}(A,m)\) be the set of points \(x\) such that the splitting \(E\oplus F\) is \(m\)-dominated along the orbit of \(x\).
The set \(\mathcal{D}_{E \oplus F}(A,m)\) is measurable, because the Oseledets splitting depends measurably on the point, and it is invariant; therefore, if \(\mu\) is ergodic, \(\mu(\mathcal{D}_{E\oplus F}(A,m))\) is either \(0\) or \(1\).

When \(E \oplus F\) is the Oseledets splitting of index \(p\), we say that the cocycle \(A\) is \emph{not \(m\)-dominated} if \(X \setminus \mathcal{D}_{E\oplus F}(A,m)\) has full measure.
In conclusion, if \(\mu\) is ergodic, then either \(E \oplus F\) is \(m\)-dominated along \(\mu\)-almost every orbit, or \(\mu\)-almost every point is eventually not \(m\)-dominated.

\section{Space exchanging sequences}

The following statement is the first ingredient for creating perturbations.
It states that, given two directions neither of which dominates the other along a sufficiently long orbit segment, there is a sequence of nearby matrices that sends one direction to the other.

\begin{prop}\label{lemaperturbativo}
Let \(M>0\) and \(\e > 0\).
 There exists a positive integer \(m_0 = m_0(M,\e) \in \N\) such that if \(m \geq m_0\) then the following property holds:

Let \(y \in X\) be a non-periodic point and \(A \in \text{\textup{C}}(X,\mathcal{G})\) such that \(\|A\| < M\) and there exist two non-collinear unit vectors \(v_0\) and \(w_0\) such that
\[
\| A^m(y)(v_0) \| < 2\|A^m(y)(w_0)\|.
\]
Then there exists a sequence $\{L_0,\ldots,L_{m-1}\}$ of length $m$, with $L_j \in \mathcal{G}$, such that $\|A(f^j(y)) - L_j\| < \e$ for $0 \leq j \leq m-1$ and\[
L_{m-1}\ldots L_0(v_0) \in \text{\textup{span}}(A^m(y)w_0).
\]
\end{prop}

\begin{proof}
Let \(A\), \(y\), \(v_0\) and \(w_0\) be as in the statement.
For every \(j \geq 0\) write \(E_j = A^j(y)(\text{\textup{span}}(v_0))\) and \(F_j = A^j(y)(\text{\textup{span}}(w_0))\).
Write
\[
\alpha_{\e} = 2\arcsin\left( \min\left\{ \frac{\e}{2M},\, \frac{1}{2} \right\} \right),
\]
so that \(\alpha_\e \in (0,\pi/3]\).
We will use repeatedly the following elementary fact: if \(R:\R^d \to \R^d\) restricts to a rotation of angle \(\alpha\) on a plane and to the identity on its orthogonal complement, then \(\|I_d - R\| = \sqrt{2-2\cos\alpha} = 2\sin(\alpha/2)\); in particular, if \(\alpha < \alpha_\e\) then
\[
\|A(f^j(y))\circ(I_d - R)\| \leq \|A\| \, \|I_d - R\| < 2M\sin(\alpha_\e/2) \leq \e .
\]
The threshold \(m_0 = m_0(M,\e)\) will be defined explicitly below.
Let \(m \geq m_0\).
We consider two cases.

\emph{First case:} Suppose there exists \(l \in \{0,1,\ldots,m-1\}\) such that
\(\measuredangle(E_l,F_l) = \alpha < \alpha_{\e}\).
Take unit vectors \(u_E \in E_l\) and \(u_F \in F_l\) such that \(\measuredangle(u_E,u_F) = \alpha\).
Let \(V = E_l \oplus F_l\) and \(W = V^{\perp}\).
Let \(\hat{R} : V \rightarrow V\) be the rotation of angle \(\alpha\) such that \(\hat{R}(u_E) = u_F\), and let \(R: \R^d \rightarrow \R^d\) be such that \(R\restrict{V}=\hat{R}\) and \(R\restrict{W} = I_{d-2}\).
By the elementary fact above, \(\|A(f^l(y)) - A(f^l(y))\circ R\| = \|A(f^l(y))\circ(I_d - R)\| < \e\).
We define \(L_l = A(f^l(y))\circ R\) and \(L_j = A(f^j(y))\) for \(j \neq l\).
By construction this sequence maps \(E_0\) onto \(F_m = \text{\textup{span}}(A^m(y)w_0)\), hence \(L_{m-1}\ldots L_0(v_0) \in \text{\textup{span}}(A^m(y)w_0)\), and the proposition follows in this case.

\emph{Second case:} In what follows we assume that \(\measuredangle(E_j,F_j) \geq \alpha_{\e}\) for every \(j = 0, \ldots, m-1\).
Recall that, by hypothesis, \(v_0\) and \(w_0\) are unit vectors satisfying \(\|A^m(y)(v_0)\| < 2 \| A^m(y)(w_0)\|\).
Write \(v_j = A^j(y)(v_0)\) and \(w_j = A^j(y)(w_0)\).

The strategy is the following. First we rotate \(v_0\) a fixed small angle towards \(w_0\).
Then, for \(1 \leq j \leq m-2\), we compose with hyperbolic perturbations contracting the direction of \(v_j\) and expanding the direction of \(w_j\).
Finally, at time \(m\), the image of \(v_0\) makes an exponentially small angle with \(\text{\textup{span}}(w_m)\) --- it is at this point, and only at this point, that the hypothesis \(\|A^m(y)(v_0)\| < 2\|A^m(y)(w_0)\|\) is used --- and we close with one last small rotation.

For every \(0 \leq j \leq m-1\) let \(\bar{v}_j = v_j/\|v_j\|\), \(\bar{w}_j = w_j/\|w_j\|\), and consider the basis of \(\R^d\) given by \(\mathcal{B}_j=\{\bar{v}_j, \bar{w}_j, e^1_{j},\ldots, e^{d-2}_j\}\), where the \(e^i_j\) form an orthonormal basis of \(\text{\textup{span}}(v_j,w_j)^{\perp}\).
Denote by \(M_j\) the change of basis matrix from the canonical basis of \(\R^d\) to \(\mathcal{B}_j\), that is, the matrix with columns \(\bar{v}_j, \bar{w}_j, e^1_j, \ldots, e^{d-2}_j\).

We claim that
\[
\|M_j\|\,\|M_j^{-1}\| = \cot\left( \frac{\measuredangle(E_j,F_j)}{2} \right) \leq \cot\left(\frac{\alpha_\e}{2}\right).
\]
Indeed, let \(\theta_j = \measuredangle(\bar{v}_j,\bar{w}_j)\), so that \(\sin\theta_j = \sin\measuredangle(E_j,F_j)\) and \(|\cos\theta_j| = \cos\measuredangle(E_j,F_j)\).
Completing \(\bar{v}_j\) to an orthonormal basis of \(\text{\textup{span}}(v_j,w_j)\) and using the \(e^i_j\), we may write \(M_j = Q_j \begin{psmallmatrix} B_j & 0 \\ 0 & I_{d-2} \end{psmallmatrix}\) with \(Q_j\) orthogonal and \(B_j = \begin{psmallmatrix} 1 & \cos\theta_j \\ 0 & \sin\theta_j \end{psmallmatrix}\).
The eigenvalues of \(B_j^{T}B_j\) are \(1 \pm |\cos\theta_j|\), hence the singular values of \(B_j\) are \(\sqrt{1+|\cos\theta_j|} \geq 1\) and \(\sqrt{1-|\cos\theta_j|} \leq 1\).
It follows that \(\|M_j\| = \|B_j\|\) and \(\|M_j^{-1}\| = \|B_j^{-1}\|\), and therefore
\[
\|M_j\|\,\|M_j^{-1}\| = \sqrt{\frac{1+|\cos\theta_j|}{1-|\cos\theta_j|}} = \frac{1+|\cos\theta_j|}{\sin\theta_j} = \cot\left(\frac{\measuredangle(E_j,F_j)}{2}\right),
\]
which proves the claim, since \(t \mapsto \cot(t/2)\) is decreasing and \(\measuredangle(E_j,F_j) \geq \alpha_\e\).

We now set
\[
C_\e = \cot\left(\frac{\alpha_\e}{2}\right), \qquad \e_1 = \frac{\e}{2MC_\e}, \qquad \alpha_1 = \frac{\alpha_\e}{2}.
\]
Note that \(C_\e\), \(\e_1\) and \(\alpha_1\) depend only on \(M\) and \(\e\); in particular they are independent of \(m\).

We denote by \(H_{j,\e_1}:\R^d \rightarrow \R^d\) the linear transformation whose matrix in the basis \(\mathcal{B}_j\) is
\[
{[H_{j,\e_1}]}_{\mathcal{B}_j}= \begin{bmatrix}
1/(1+\e_1) & 0 & 0 \\
0 & 1+\e_1 & 0  \\
0 & 0 & I_{d-2}
\end{bmatrix}.
\]
Observe that, since \(v_j\) and \(w_j\) are eigenvectors, \(H_{j,\e_1}(v_j) = \frac{1}{1+\e_1}v_j\) and \(H_{j,\e_1}(w_j) = (1+\e_1)w_j\); rescaling the eigenvectors does not change the matrix above.

Denote \(\theta_0 = \measuredangle(v_0,w_0) \in [\alpha_\e, \pi-\alpha_\e]\) and \(Y = \text{\textup{span}}(v_0,w_0)\).
Let \(\hat{R}_0 : Y \rightarrow Y\) be the rotation of angle \(\alpha_1\) in direction from \(v_0\) to \(w_0\), extended to \(R_0:\R^d \to \R^d\) by the identity on \(Y^\perp\), and define \(L_0 = A(y)R_0\).
Since \(\alpha_1 < \alpha_\e\), we get \(\|A(y) - L_0\| < \e\).
Writing \(w_0 = \cos\theta_0\, v_0 + \sin\theta_0\, v_0^{\perp}\) with \(v_0^\perp \in Y\) unitary and orthogonal to \(v_0\), a direct computation gives
\[
R_0(v_0) = a\,v_0 + b\,w_0, \qquad a = \frac{\sin(\theta_0 - \alpha_1)}{\sin\theta_0}, \quad b = \frac{\sin\alpha_1}{\sin\theta_0},
\]
and \(a, b > 0\) because \(0 < \alpha_1 < \alpha_\e \leq \theta_0 < \pi\).

For \(1 \leq j \leq m-2\) define \(L_j = A(f^j(y))H_{j,\e_1}\).
Note that \(\| I_d - {[H_{j,\e_1}]}_{\mathcal{B}_j}\| = \max\left\{\frac{\e_1}{1+\e_1},\, \e_1\right\} = \e_1\), hence
\begin{align*}
\|A(f^j(y)) - L_j\| &\leq \|A\|\, \|I_d - H_{j,\e_1}\| \leq \|A\|\, \|M_j\| \,\| I_d - {[H_{j,\e_1}]}_{\mathcal{B}_j} \| \,\|M^{-1}_j\| \\
&\leq M C_{\e}\,\e_1 = \frac{\e}{2} < \e.
\end{align*}

Denoting \(\hat{L}_j = L_j \circ \cdots \circ L_0\), an immediate induction using the cocycle relation \(A(f^j(y))(v_j) = v_{j+1}\) (and likewise for \(w_j\)) gives, for every \(0 \leq j \leq m-2\),
\[
\hat{L}_j(v_0) = a\,{(1+\e_1)}^{-j}\,v_{j+1} + b\,{(1+\e_1)}^{j}\,w_{j+1}.
\]

Applying \(A(f^{m-1}(y))\) to \(\hat{L}_{m-2}(v_0)\) we obtain the vector
\[
u = c_1 v_m + c_2 w_m, \qquad c_1 = a\,{(1+\e_1)}^{-(m-2)}, \quad c_2 = b\,{(1+\e_1)}^{m-2},
\]
with \(c_1, c_2 > 0\). Let \(\beta = \measuredangle\!\left(u, \text{\textup{span}}(w_m)\right)\) and \(\theta_m = \measuredangle(v_m,w_m) \in (0,\pi)\) (note that no lower bound on \(\theta_m\) is required).
Since the orthogonal projection onto \(\text{\textup{span}}(w_m)^\perp\) kills \(c_2 w_m\), and \(\|u\|\) is at least the norm of the projection of \(u\) onto \(\text{\textup{span}}(v_m)^\perp\),
\[
\sin\beta = \frac{c_1 \|v_m\| \sin\theta_m}{\|u\|} \leq \frac{c_1 \|v_m\| \sin\theta_m}{c_2 \|w_m\| \sin\theta_m} = \frac{c_1 \|v_m\|}{c_2 \|w_m\|}.
\]
By the hypothesis \(\|v_m\| = \|A^m(y)(v_0)\| < 2\|A^m(y)(w_0)\| = 2\|w_m\|\), and since \(a/b = \sin(\theta_0-\alpha_1)/\sin\alpha_1 \leq 1/\sin\alpha_1\),
\[
\sin\beta < \frac{2a}{b}\,{(1+\e_1)}^{-2(m-2)} \leq \frac{2}{\sin\alpha_1}\,{(1+\e_1)}^{-2(m-2)}.
\]

We finally define
\[
m_0 = m_0(M,\e) = 2 + \left\lceil \frac{\log\!\big( 2/(\sin\alpha_1 \sin\alpha_\e) \big)}{2\log(1+\e_1)} \right\rceil,
\]
which depends only on \(M\) and \(\e\) since \(\alpha_\e, \alpha_1\) and \(\e_1\) do.
For \(m \geq m_0\) the previous estimate gives \(\sin\beta < \sin\alpha_\e\), and since \(\beta \in [0,\pi/2]\) and \(\alpha_\e \leq \pi/3\), we conclude \(\beta < \alpha_\e\).
Moreover \(c_1\|v_m\| < c_2\|w_m\|\) implies \(\langle u, w_m\rangle \geq c_2\|w_m\|^2 - c_1\|v_m\|\|w_m\| > 0\), so \(u\) makes an acute angle, equal to \(\beta\), with the vector \(w_m\).

Let \(R':\R^d \to \R^d\) restrict to the rotation of angle \(\beta\) of \(\text{\textup{span}}(v_m,w_m)\) taking the direction of \(u\) to the direction of \(w_m\), and to the identity on the orthogonal complement, and define
\[
L_{m-1} = R' \circ A(f^{m-1}(y)).
\]
Since \(\beta < \alpha_\e\),
\[
\|A(f^{m-1}(y)) - L_{m-1}\| = \|(I_d - R')\circ A(f^{m-1}(y))\| \leq \|I_d - R'\|\,\|A\| < 2M\sin(\alpha_\e/2) \leq \e.
\]
Finally,
\[
L_{m-1}\cdots L_0(v_0) = R'(u) \in \text{\textup{span}}(w_m) = \text{\textup{span}}(A^m(y)(w_0)),
\]
as required.

Every \(L_j\) constructed in either case belongs to \(\mathcal{G}\), since it is the product of \(A(f^j(y)) \in \mathcal{G}\) with a rotation or with one of the matrices \(H_{j,\e_1}\), all of which have determinant one. This finishes the proof.
\end{proof}

\begin{defi}
A sequence \(\{ L_0,\ldots,L_{m-1}\}\) of matrices in \(\mathcal{G}\) such that \(\|A(f^j(y))-L_j\| < \e\) for \(0 \leq j \leq m-1\) and \(L_{m-1}\cdots L_0(v_0) \in A^{m}(y)(\text{\textup{span}}(w_0))\) for two non-collinear vectors \(v_0, w_0\), as in Proposition~\ref{lemaperturbativo}, is called a \emph{space exchanging \(\e\)-sequence over \(y\) of length \(m\)}.
\end{defi}

Recall that \(\reg(A)\) denotes the set of regular points. For \(x \in \reg(A)\), let \(s\) be given by \(\gamma_s(A) = \lambda_p(A)\) and \(\gamma_{s+1}(A) = \lambda_{p+1}(A)\), and write \(E(x) = \bigoplus_{i=1}^s E_i(x)\) and \(F(x)= \bigoplus_{i=s+1}^{k(A)} E_i(x)\).
This is the Oseledets splitting \(E \oplus F\) of index \(p\) introduced in~\eqref{split}.

If \(A\) is not \(m\)-dominated then, discarding a set of measure zero, we may assume that every point of \(\reg=\reg(A)\) is eventually not \(m\)-dominated and is not periodic.
Thus every \(x \in \reg\) has an iterate \(y=f^n(x)\), \(n\in\Z\), satisfying the non-\(m\)-domination condition~\eqref{nondom}:
\[
\frac{\| A^m(y)\restrict{F(y)} \|}{\text{\textup{\textbf{m}}}(A^m(y)\restrict{E(y)})} > \frac{1}{2}.
\]

At such a point \(y\), which is not periodic, the non-\(m\)-domination condition provides unit vectors \(v \in E(y)\), \(w \in F(y)\) --- non-collinear, since \(E(y) \cap F(y) = \{0\}\) --- such that \(\|A^m(y)(v)\| < 2 \| A^m(y)(w)\|\): it suffices to take \(v\) realising \(\text{\textup{\textbf{m}}}(A^m(y)\restrict{E(y)})\) and \(w\) realising \(\| A^m(y)\restrict{F(y)} \|\).
Hence the hypotheses of Proposition~\ref{lemaperturbativo} are satisfied, and space exchanging \(\e\)-sequences over \(y\) exist.

Recall that \(\Lambda_i(A) = \lambda_1(A) + \cdots + \lambda_i(A)\) denotes the sum of the \(i\) largest Lyapunov exponents of \(A\).

The following lemma states that, if \(E\oplus F\) is not dominated, there are long space exchanging \(\e\)-sequences whose products have a small \(p\)-th exterior power.

The lemma is a slight modification of~\cite[Proposition~4.2]{BochiViana_2005}, and the proof given there works in our setting without changes, so we do not reproduce it.
Instead, after the statement we explain how the space exchanging sequences of Proposition~\ref{lemaperturbativo} enter that proof, since the conclusions of the two statements are phrased in different terms.

\begin{lema}\label{estimacionnorma}
Let \(M>0,\ \e > 0,\ \delta > 0\).
Let \(m \in \N\) be such that \(m \geq m_0(M,\e)\), and let \(A \in \text{\textup{C}}(X,\mathcal{G})\) be such that \(\| A\| < M\) and its Oseledets splitting \(E\oplus F\) of index \(p\) (so that \(\lambda_p(A) > \lambda_{p+1}(A)\)) is not \(m\)-dominated.
Then there exists a measurable function \(N_A: \reg \rightarrow \N\) with the following property: for almost every \(x \in \reg\) and every \(n \geq N_A(x)\) there exists a space exchanging \(\e\)-sequence \(\{\hat{L}_0^{(x,n)},\ldots, \hat{L}_{n-1}^{(x,n)}\}\) over \(x\) of length \(n\) such that
\[
 \frac{1}{n}\log \|\wedge^p (\hat{L}_{n-1}^{(x,n)}\cdots \hat{L}_0^{(x,n)}) \| \leq \frac{\Lambda_{p-1}(A) + \Lambda_{p+1}(A)}{2} + \delta.
\]
\end{lema}
\begin{rem}
The function \(N_A\) depends on the cocycle \(A\) and on \(\e\), \(\delta\) and \(m\).
\end{rem}

The sequence of Lemma~\ref{estimacionnorma} is obtained by concatenation.
Let \(C \subset \reg\) be the set of points where the non-\(m\)-domination condition~\eqref{nondom} holds; since almost every regular point is eventually not \(m\)-dominated, \(\reg = \bigcup_{n \in \Z} f^n(C)\) up to a zero measure set.
The recurrence lemma~\cite[Lemma~3.12]{BOCHI_2002a}, applied to the set \(C\) and combined with the convergence times given by Oseledets' theorem, provides the measurable function \(N_A\) and, for almost every \(x \in \reg\) and every \(n \geq N_A(x)\), an iterate \(y = f^{\ell}(x) \in C\) with \(\ell/n\) close to \(1/2\).
The space exchanging \(\e\)-sequence over \(x\) of length \(n\) is then the concatenation
\[
\{A(f^i(x)) :\ 0 \leq i < \ell\}, \qquad \{L_0,\ldots,L_{m-1}\}, \qquad \{A(f^i(x)) :\ \ell + m \leq i < n\},
\]
where \(\{L_0,\ldots,L_{m-1}\}\) is the sequence over \(y\) given by Proposition~\ref{lemaperturbativo}, applied to unit vectors \(v_0 \in E(y)\) and \(w_0 \in F(y)\) realising the non-\(m\)-domination condition, as explained after the Definition above.

The norm estimate is carried out, as in~\cite{BochiViana_2005}, on the \(p\)-th exterior power, with respect to the splitting \(\wedge^p \R^d = V \oplus H\): here \(V = \wedge^p E\) is the one-dimensional Oseledets subspace of \(\wedge^p A\) associated with the top exponent \(\lambda_1^{\wedge p}(A) = \Lambda_p(A)\), and \(H\) is the sum of the remaining Oseledets subspaces of \(\wedge^p A\), whose largest exponent is \(\lambda_2^{\wedge p}(A) = \Lambda_{p-1}(A) + \lambda_{p+1}(A)\).
Since the conclusion of Proposition~\ref{lemaperturbativo} is phrased in terms of a single vector, while the proof of~\cite[Proposition~4.2]{BochiViana_2005} uses the inserted block through its action on \(\wedge^p \R^d\), the bridge between the two statements is the following observation:
\begin{equation}\label{puente}
\wedge^p(L_{m-1}\cdots L_0)\left(V(y)\right) \subset H(f^m(y)).
\end{equation}
Indeed, complete \(v_0\) to a basis \(\{v_0, u_2, \ldots, u_p\}\) of \(E(y)\), so that \(V(y)\) is spanned by \(v_0 \wedge u_2 \wedge \cdots \wedge u_p\).
Writing \(L = L_{m-1}\cdots L_0\), the image of this \(p\)-vector is \(w \wedge L(u_2) \wedge \cdots \wedge L(u_p)\), where \(w = L(v_0) \in \text{\textup{span}}(A^m(y)w_0) \subset F(f^m(y))\) by Proposition~\ref{lemaperturbativo} and the invariance of \(F\).
Expanding each factor in a basis of \(\R^d\) adapted to the Oseledets splitting at \(f^m(y)\), every resulting elementary \(p\)-vector contains, coming from the expansion of \(w\), at least one basis vector of \(F(f^m(y))\); all such elementary \(p\)-vectors belong to \(H(f^m(y))\), and the inclusion~\eqref{puente} follows.

With the observation~\eqref{puente}, the proof of~\cite[Proposition~4.2]{BochiViana_2005} applies verbatim to the concatenated sequence.
In the block decomposition relative to \(V \oplus H\), the two outer pieces are block-diagonal, since \(V\) and \(H\) are \(\wedge^p A\)-invariant, and for \(n\) large their blocks grow at rates at most \(\tfrac{1}{2}\left(\lambda_1^{\wedge p}(A) + \delta\right)\) and \(\tfrac{1}{2}\left(\lambda_2^{\wedge p}(A) + \delta\right)\) respectively (the outer pieces have length only approximately \(n/2\); replacing their lengths by \(n/2\) costs a term of order \(n\delta\), provided that the accuracy of the recurrence time given by~\cite[Lemma~3.12]{BOCHI_2002a} is chosen in terms of a bound for \(|\lambda_2^{\wedge p}(A)|\), which may be negative, and not only for \(\lambda_1^{\wedge p}(A)\)), while the inserted block has norm bounded independently of \(n\) and, by~\eqref{puente}, vanishing \(V\)-to-\(V\) entry.
Hence every entry of the product grows at rate at most \(\tfrac{1}{2}\left(\lambda_1^{\wedge p}(A) + \lambda_2^{\wedge p}(A)\right) = \tfrac{1}{2}\left(\Lambda_{p-1}(A) + \Lambda_{p+1}(A)\right)\), up to a term of order \(\delta\), and the estimate of the Lemma follows, since the comparison between the entrywise norm and the operator norm involves only the angle \(\measuredangle(V,H)\), whose sine decays subexponentially along regular orbits.

\section{Lowering the sum of the largest Lyapunov exponents}

In the previous section we showed that, in the absence of \(m\)-domination for a large \(m\), there are long space exchanging \(\e\)-sequences along the orbit of almost every point whose products have a small \(p\)-th exterior power.
Our aim is to realise these sequences as perturbations of the cocycle, so as to lower the sum of its \(p\) largest Lyapunov exponents.

The main ingredient is the following perturbation result.
 Recall that \(\mathcal{R}\) contains no periodic points.

We first fix some conventions that are used in the rest of the paper.
 Given \(M \geq 2\) and \(\e > 0\) we write
\begin{equation}\label{defepsM}
\e_M = \frac{\min\{\e,1\}}{8M^2}.
\end{equation}
The number \(\e_M\) is the size of the space exchanging sequences that we realise as perturbations of size \(\e\); the loss from \(\e\) to \(\e_M\) comes from the interpolation step in the proof of Proposition~\ref{mainperturbativo}, where the matrices have to be kept in \(\mathcal{G}\).
 Lemma~\ref{estimacionnorma}, applied with parameters \(\e' > 0\), \(\delta > 0\) and \(m \geq m_0(M,\e')\), provides a full measure set \(\reg' \subset \reg\) and a measurable function \(N_A:\reg' \rightarrow \N\) such that its conclusion holds for \emph{every} \(x \in \reg'\) and every \(n \geq N_A(x)\).
 In this section Lemma~\ref{estimacionnorma} is always applied with \(\e' = \e_M\); the set \(\reg'\) and the function \(N_A\) refer to this choice, and they also depend on \(\delta\) and \(m\).
 Finally, given \(A \in \text{\textup{C}}(X,\mathcal{G})\), we denote by \(\delta_A > 0\) a modulus of uniform continuity of \(A\) for \(\e_M\), that is, a number such that \(d(y,z) < \delta_A\) implies \(\|A(y) - A(z)\| < \e_M\).

\begin{prop}\label{mainperturbativo}
Let \(M \geq 2\), \(\e > 0\) and \(\delta>0\).
 Let \(A \in \text{\textup{C}}(X,\mathcal{G})\) be such that \(\|A\| < M\), \(\|A^{-1}\| < M\), \(\lambda_p(A) > \lambda_{p+1}(A)\), and the Oseledets splitting \(E \oplus F\) of index \(p\) is not \(m\)-dominated for some \(m \geq m_0(M,\e_M)\).
 Let \(x \in \reg'\), \(n \geq N_A(x)\), and let \(V \subset X\) be an open set such that
\begin{enumerate}
\item[(i)] the sets \(f^{l}(V)\), \(0 \leq l \leq n-1\), are pairwise disjoint;
\item[(ii)] \(f^{l}(V) \subset B\big(f^{l}(x), \delta_{A}\big)\) for every \(0 \leq l \leq n-1\).
\end{enumerate}
Then for every compact set \(K \subset V\) there exists a cocycle \(B \in \text{\textup{C}}(X,\mathcal{G})\) such that:
\begin{enumerate}
\item \(\| A - B \| < \e\);
\item \(B\) is constant on \(f^{l}(K)\) for every \(l = 0,1,\ldots,n-1\);
\item \(B = A\) on \(X \setminus \bigcup_{l=0}^{n-1} f^{l}(V)\);
\item for every \(y \in K\) it holds that
\[
\frac{1}{n}\log \| \wedge^p B^n(y) \| \leq \frac{\Lambda_{p-1}(A) + \Lambda_{p+1}(A)}{2} + \delta .
\]
\end{enumerate}
\end{prop}

\begin{rem}
Open sets \(V\) as in the statement always exist: since \(x\) is not periodic and \(f\) is a homeomorphism, there is an open set \(U \ni x\) with \(f^{i}(U) \cap f^{j}(U) = \emptyset\) for \(0 \leq i < j \leq n-1\), and
\(V = U \cap \bigcap_{l=0}^{n-1} f^{-l}\big(B(f^{l}(x),\delta_{A})\big)\)
is an open neighbourhood of \(x\) satisfying (i) and (ii).
 No connectedness is involved; this matters because in a totally disconnected space \(X\), such as a shift space, the connected components of open sets need not be open.
\end{rem}

\begin{proof}
Let \(V\) and \(K\) be as in the statement.
 If \(\partial V \neq \emptyset\), the compactness of \(K\) and \(\partial V\) gives \(\rho = d(K, \partial V) > 0\) and we set \(V_0 = \{ y \in V:\ d(y,K) < \rho/2 \}\); if \(\partial V = \emptyset\) we set \(V_0 = V\).
 In both cases \(V_0\) is open and \(K \subset V_0 \subset \overline{V_0} \subset V\).
 Since \(f\) is a homeomorphism, \(\hat{V} = \bigcup_{l=0}^{n-1} f^{l}(V_0)\) is open, \(\hat{K} = \bigcup_{l=0}^{n-1} f^{l}(K)\) is compact and \(\overline{\hat{V}} \subset \bigcup_{l=0}^{n-1} f^{l}(V)\).
 By Urysohn's lemma there exists a continuous function \(h:X \rightarrow [0,1]\) such that \(h \equiv 0\) on \(X \setminus \hat{V}\) and \(h \equiv 1\) on \(\hat{K}\).

By Lemma~\ref{estimacionnorma}, applied with \(\e_M\) and \(\delta\) (recall that \(m \geq m_0(M,\e_M)\), \(x \in \reg'\) and \(n \geq N_A(x)\)), there is a space exchanging \(\e_M\)-sequence of length \(n\) over \(x\), namely \({\{\hat{L}_l\}}_{l=0}^{n-1}\), with \(\hat{L}_l \in \mathcal{G}\), \(\|\hat{L}_l - A(f^{l}(x))\| < \e_M\) and
\begin{equation}\label{normasecuencia}
\frac{1}{n}\log \|\wedge^p (\hat{L}_{n-1} \cdots \hat{L}_0)\| \leq \frac{\Lambda_{p-1}(A) + \Lambda_{p+1}(A)}{2} + \delta .
\end{equation}

\smallskip
Fix \(l \in \{0,\ldots,n-1\}\) and write \(A_l = A(f^{l}(x))\).
 For \(y \in f^{l}(V)\) and \(t \in [0,1]\) let
\[
D_l(y,t) = (1-t)\,A(y) + t\,\hat{L}_l .
\]
By (ii), \(\|A(y) - A_l\| < \e_M\), hence \(\|D_l(y,t) - A_l\| < \e_M\) and \(\|D_l(y,t) - A(y)\| \leq t\,\|\hat{L}_l - A(y)\| < 2\e_M\).
 Write \(D_l(y,t) = A_l\,(I_d + E)\) with \(E = A_l^{-1}\big(D_l(y,t) - A_l\big)\); since \(\|A_l^{-1}\| < M\),
\[
\|E\| < M \e_M = \frac{\min\{\e,1\}}{8M} \leq \frac{1}{16} .
\]
In particular \(D_l(y,t)\) is invertible.
 Moreover, the singular values of \(I_d + E\) lie in \([1 - \|E\|, 1 + \|E\|]\), and \(\det(I_d + sE) \neq 0\) for every \(s \in [0,1]\), so that \(\det(I_d + E) > 0\); hence
\begin{equation}\label{detD}
\det D_l(y,t) = \det A_l \cdot \det(I_d + E), \qquad \det(I_d + E) \in \big[{(1 - \|E\|)}^{d}, {(1 + \|E\|)}^{d}\big].
\end{equation}
Thus \(D_l(y,t) \in \gld\), and \(\det D_l(y,t)\) has the sign of \(\det A_l\).

We now define \(B\) on \(\bigcup_{l=0}^{n-1} f^{l}(V)\), the sets \(f^{l}(V)\) being pairwise disjoint by (i).
 If \(\mathcal{G} = \gld\), set \(B(y) = D_l(y,h(y))\) for \(y \in f^{l}(V)\).
 If \(\mathcal{G} = \sld\), then \(\det A_l = 1\), so \(\det D_l(y,t) > 0\) by~\eqref{detD}, and we set
\[
B(y) = {\det\!\big(D_l(y,h(y))\big)}^{-1/d}\, D_l(y,h(y)) \in \sld \qquad \text{for } y \in f^{l}(V).
\]
In both cases we set \(B(y) = A(y)\) for \(y \in X \setminus \bigcup_{l=0}^{n-1} f^{l}(V)\).
 Note that \(B(y) = A(y)\) whenever \(h(y) = 0\) (for \(\mathcal{G} = \sld\) because \(\det A(y) = 1\)), and that \(B(y) = \hat{L}_l\) whenever \(y \in f^{l}(V)\) and \(h(y) = 1\) (for \(\mathcal{G} = \sld\) because \(\det \hat{L}_l = 1\)).

\smallskip
On the open set \(\bigcup_{l} f^{l}(V)\) the map \(B\) is continuous, since on each \(f^{l}(V)\) it is obtained from the continuous maps \(y \mapsto A(y)\) and \(y \mapsto h(y)\) by the continuous operations above (the determinant does not vanish on the relevant set).
 On the open set \(X \setminus \overline{\hat{V}}\) we have \(h \equiv 0\), hence \(B = A\) there, which is continuous.
 These two open sets cover \(X\), because \(\overline{\hat{V}} \subset \bigcup_l f^{l}(V)\); therefore \(B\) is continuous on \(X\), and it takes values in \(\mathcal{G}\), that is, \(B \in \text{\textup{C}}(X,\mathcal{G})\).
 Item 3 holds by definition, and item 2 holds because \(h \equiv 1\) on \(f^{l}(K) \subset \hat{K}\), so that \(B \equiv \hat{L}_l\) on \(f^{l}(K)\).
 For item 4, let \(y \in K\); then \(f^{l}(y) \in f^{l}(K)\) and \(B(f^{l}(y)) = \hat{L}_l\) for \(0 \leq l \leq n-1\), so \(B^n(y) = \hat{L}_{n-1} \cdots \hat{L}_0\), and item 4 is~\eqref{normasecuencia}.

\smallskip
Lastly, let \(y \in f^{l}(V)\) and write \(D = D_l(y,h(y))\), so that \(\|D - A(y)\| < 2\e_M = \min\{\e,1\}/(4M^2) \leq \e/16\).
 If \(\mathcal{G} = \gld\), then \(\|B(y) - A(y)\| = \|D - A(y)\| < \e\).
 If \(\mathcal{G} = \sld\), let \(s = {(\det D)}^{-1/d}\).
 By~\eqref{detD}, \(s \in \big[{(1+\|E\|)}^{-1}, {(1-\|E\|)}^{-1}\big]\), so \(|s - 1| \leq \|E\|/(1 - \|E\|) < \tfrac{16}{15}\cdot\tfrac{\min\{\e,1\}}{8M} < \tfrac{\e}{7M}\).
 Since \(\|D\| \leq \|A_l\| + \e_M < M + 1\) and \(M \geq 2\),
\[
\|B(y) - A(y)\| \leq |s-1|\,\|D\| + \|D - A(y)\| < \frac{\e}{7M}\,(M+1) + \frac{\e}{16} \leq \frac{3\e}{14} + \frac{\e}{16} < \e .
\]
Together with \(B = A\) outside \(\bigcup_l f^{l}(V)\), this gives \(\|A - B\| < \e\) and finishes the proof.
\end{proof}

We now combine many perturbations as in Proposition~\ref{mainperturbativo} to obtain a nearby cocycle with a smaller value of \(\Lambda_p\).
 The following construction will be useful.

\begin{defi}
Let \( f \in \Hom[\mu]{X}\).
 A \emph{tower} is a pair of measurable sets  \( (T,T_{b}) \) such that there is a positive integer \(n \), called the \emph{height} of the tower, such that the sets \(T_b,f(T_b),\ldots, f^{n-1}(T_b)\) are pairwise disjoint and their union is \(T\).
 The set \(T_b\) is called the \emph{base} of the tower.
The set \(f^{i}(T_b)\) is called the \(i\)-th \emph{floor} of the tower.

A \emph{castle} is a pair of measurable sets \( (Q, Q_b) \) such that there exists a finite or countable family of pairwise disjoint towers \( (T^i, T^i_b) \) such that \(Q = \bigcup_{i} T^i \) and \(Q_b = \bigcup_i T_b^i\).
 We call \(Q_b\) the \emph{base} of the castle.
\end{defi}

Given a set \(G \subset X\) of positive measure, consider the first return time \(\tau : G \rightarrow \N\), \(\tau(x) = \min\{ n \geq 1:\ f^n(x) \in G\}\), which is defined almost everywhere on \(G\) by Poincar\'e's recurrence theorem.
Denoting \(G_n = \tau^{-1}(n)\) and \(T_n = G_n \cup f(G_n) \cup \cdots \cup f^{n-1}(G_n)\), each \((T_n, G_n)\) is a tower.
The castle \((Q,G)\) with base \(G\) given by the union of the towers \(T_n\) is called the \emph{Kakutani--Rokhlin castle} with base \(G\).
Note that \(Q = \bigcup_{n \in \Z}f^n(G)\) up to a set of measure zero; in particular, \(Q\) is invariant.

The following lemma shows that, if the boundary of \(G\) has measure zero, then so do the boundaries of the bases of the towers; hence, up to a set of measure zero, we may replace each \(G_i\) by its interior \(\inte{G_i}\) and assume that the towers have open bases.

\begin{lema}\label{baseabs}
Let \(G\) be a measurable set such that \(\mu(\partial G) = 0\).
 Consider the Kakutani--Rokhlin castle \((Q,G)\) and its towers \((T_i,G_i)\).
Then \(\mu(\partial G_i) = 0\) for every \(i \in \N\).
\end{lema}

\begin{proof}
For every \(i \in \N\) the base \(G_i = \tau^{-1}(i)\) is the set of points \(x \in G\) such that \(f^{i}(x) \in G\) and \(f^{k}(x) \notin G\) for \(1 \leq k \leq i-1\); that is,
\[
G_i = \big( G \cap f^{-i}(G) \big) \setminus \bigcup_{k=1}^{i-1} f^{-k}(G),
\]
a finite Boolean combination of the sets \(G, f^{-1}(G), \ldots, f^{-i}(G)\).
 Since \(\partial(S \cap T)\), \(\partial(S \cup T)\) and \(\partial(S \setminus T)\) are contained in \(\partial S \cup \partial T\) for any sets \(S, T \subset X\), and \(\partial\big(f^{-k}(G)\big) = f^{-k}(\partial G)\) because \(f\) is a homeomorphism, we obtain
\[
\partial G_i \subset \partial G \cup \bigcup_{k=1}^{i} f^{-k}(\partial G).
\]
Every set on the right hand side has measure zero, because \(\mu(\partial G) = 0\) and \(\mu\) is \(f\)-invariant.
 Hence \(\mu(\partial G_i) = 0\).
\end{proof}

Our main tool for constructing sets whose boundary has measure zero are open balls, as the following remark shows.

\begin{obs}\label{bolas} Let \(x \in X\). Then all but countably many radii \(r > 0\) satisfy \(\mu(\partial B(x,r)) = 0\).

\end{obs}
Indeed, the spheres \(\{y \in X:\ d(x,y) = r\}\), \(r > 0\), which contain the boundaries \(\partial B(x,r)\), are pairwise disjoint.
Since \(\mu\) is finite, at most countably many of them have positive measure.

We are ready to globalise the perturbation of Proposition~\ref{mainperturbativo} in order to get a cocycle with a lower sum of the first \(p\) Lyapunov exponents.
 Recall the notation \(\e_M = \min\{\e,1\}/(8M^2)\) introduced in~\eqref{defepsM}.

\begin{teo}\label{bajaexponente}
Let \(M \geq 2\), \(\e > 0\) and \(\delta >0\).
 Let \(A \in \text{\textup{C}}(X,\mathcal{G})\) be such that \(\|A\| < M\), \(\|A^{-1}\| < M\), \(\lambda_p(A) > \lambda_{p+1}(A)\), and the Oseledets splitting \(E\oplus F\) of index \(p\) is not \(m\)-dominated for some \(m \geq m_0(M,\e_M)\).
 Then there exists a cocycle \(B \in \text{\textup{C}}(X,\mathcal{G})\) such that:
\begin{enumerate}
\item \( \| A - B \| < \e \),
\item \( \Lambda_p (B) < \dfrac{\Lambda_{p-1}(A) + \Lambda_{p+1} (A)}{2}+ \delta\).
\end{enumerate}
\end{teo}

\begin{proof}
Since the conclusion for a smaller \(\delta\) implies the conclusion for \(\delta\), we may assume \(\delta < 1\).
 Write
\[
\Lambda' = \frac{\Lambda_{p-1}(A) + \Lambda_{p+1}(A)}{2}.
\]
For the sake of clarity we divide the proof in two parts.

\medskip
\noindent
\emph{Part 1: Construction of the cocycle \(B\).}
\medskip

Let \(c > 1\) be defined by
\[
\log c = \max\big\{ p\log (M + \e),\ |\Lambda'|,\ 1 \big\}.
\]
If \(B \in \text{\textup{C}}(X,\mathcal{G})\) satisfies \(\|A - B\| < \e\), then \(\|B(z)\| < M + \e\) for every \(z \in X\); since the operator norm of \(\wedge^p L\) is the product of the \(p\) largest singular values of \(L\), \(\|\wedge^p B(z)\| \leq \|B(z)\|^p < c\), and therefore
\begin{equation}\label{cotac}
\|\wedge^p B^{t}(z)\| \leq c^{\,t} \qquad \text{for every } z \in X \text{ and every } t \geq 0 .
\end{equation}
(Note that \(\wedge^p B\) need not belong to \(\mathcal{U}(\wedge^p A, \e)\); the bound \eqref{cotac} only uses \(\|\wedge^p L\| \leq \|L\|^p\).)
 Set
\begin{equation}\label{defeps1}
\e_1 = \frac{\delta}{12 \log c}, \qquad \text{so that} \qquad 6\,\e_1 \log c = \frac{\delta}{2} \quad \text{and} \quad 3\e_1 < \frac{1}{4}.
\end{equation}

Let \(\reg'\) and \(N_A\) be as in Proposition~\ref{mainperturbativo} for the parameters \(\e\) and \(\delta/2\) (that is, \(N_A\) is the function of Lemma~\ref{estimacionnorma} for \(\e_M\), \(\delta/2\) and \(m\)).
 For every positive integer \(N\) let
\[
\mathcal{N}_N = \{ y \in \reg':\ N_A(y) \leq N \}.
\]
These sets form a nondecreasing sequence whose union is \(\reg'\), which has full measure; hence we can fix \(N_0 \in \N\) such that
\begin{equation}\label{medidaN0}
\mu(\mathcal{N}_{N_0}) \geq 1 - \tfrac{9}{10}\,\e_1 .
\end{equation}
Let \(\delta_{A}\) be the modulus of uniform continuity of \(A\) for \(\e_M\) fixed before Proposition~\ref{mainperturbativo}.

Fix \(x_0 \in \reg\).
 As \(x_0\) is not periodic and \(f\) is a homeomorphism, there is \(r > 0\) such that the sets \(f^{i}(B(x_0,r))\), \(0 \leq i \leq 10N_0 - 1\), are pairwise disjoint; by Remark~\ref{bolas} we may choose \(r\) so that, in addition, the ball \(G = B(x_0,r)\) satisfies \(\mu(\partial G) = 0\).
 Since \(\mu\) has full support, \(\mu(G) > 0\).
 Consider the Kakutani--Rokhlin castle \((Q,G)\) with base \(G\) and its towers \((T_i,G_i)\), where \(G_i = \tau^{-1}(i)\) and \(T_i = \bigcup_{l=0}^{i-1} f^{l}(G_i)\).
 By the choice of \(r\), the return time satisfies \(\tau \geq 10N_0\) on \(G\), so \(G_i = \emptyset\) for \(i < 10N_0\).
 As noted after the definition of the Kakutani--Rokhlin castle, \(Q\) coincides with the invariant set \(\bigcup_{n \in \Z} f^{n}(G)\) up to a null set; since \(\mu(G) > 0\) and \(\mu\) is ergodic, \(\mu(Q) = 1\), and thus
\begin{equation}\label{kac}
\sum_{i \geq 10N_0} i\,\mu(G_i) = \sum_{i \geq 10N_0} \mu(T_i) = \mu(Q) = 1 .
\end{equation}
By Lemma~\ref{baseabs}, \(\mu(\partial G_i) = 0\) for every \(i\); replacing each \(G_i\) by its interior changes every tower by a null set, so from now on we assume that every \(G_i\) is open with \(\mu(\partial G_i) = 0\) (and \eqref{kac} still holds).
 By \eqref{kac} there is an integer \(H > 10N_0\) such that
\begin{equation}\label{colatorres}
\mu\Big( \bigcup_{i \geq H} T_i \Big) < \e_1 .
\end{equation}

By uniform continuity of \(f, f^{2}, \ldots, f^{H}\) on the compact space \(X\), there is \(\eta > 0\) such that
\begin{equation}\label{etaunif}
d(y,z) < \eta \quad \Longrightarrow \quad d\big(f^{l}(y), f^{l}(z)\big) < \delta_{A} \qquad \text{for every } 0 \leq l \leq H .
\end{equation}
Fix \(i \in \{10N_0, \ldots, H-1\}\).
 By Remark~\ref{bolas}, every point of \(\overline{G_i}\) is the centre of an open ball of radius less than \(\eta/2\) with null boundary; by compactness of \(\overline{G_i}\), finitely many of these balls, say \(B_1, \ldots, B_{n(i)}\), cover \(\overline{G_i}\).
 Define
\[
G_{i,j} = \inte{\Big( B_j \setminus \bigcup_{l=1}^{j-1} B_l \Big) \cap G_i}, \qquad j = 1, \ldots, n(i).
\]
The sets \(G_{i,j}\) are open, pairwise disjoint, contained in \(G_i\), and each has diameter less than \(\eta\).
 Since the boundary of a finite Boolean combination of sets is contained in the union of their boundaries, and \(\partial \inte{S} \subset \partial S\) for every \(S\), we have \(\mu(\partial G_{i,j}) = 0\); and since the sets \((B_j \setminus \bigcup_{l<j} B_l) \cap G_i\), \(1 \leq j \leq n(i)\), partition \(G_i\) and differ from their interiors by subsets of their boundaries,
\begin{equation}\label{particionnula}
\mu\Big( G_i \setminus \bigcup_{j=1}^{n(i)} G_{i,j} \Big) = 0 .
\end{equation}
Let \(T_{i,j} = \bigcup_{l=0}^{i-1} f^{l}(G_{i,j})\).
 These subtowers of \(T_i\) are pairwise disjoint and, by \eqref{particionnula} and the invariance of \(\mu\), they cover \(T_i\) up to a null set.

For \(i \in \{10N_0,\ldots,H-1\}\) and \(1 \leq j \leq n(i)\) let
\[
k(i,j) = \min\big\{ k \in \{0,\ldots,i - N_0\}:\ f^{k}(G_{i,j}) \cap \mathcal{N}_{N_0} \neq \emptyset \big\}
\]
whenever this set is nonempty, and let \(\mathcal{J}\) be the (finite) set of pairs \((i,j)\) for which \(k(i,j)\) is defined.
 For \((i,j) \in \mathcal{J}\) set
\begin{gather*}
n(i,j) = i - k(i,j) \geq N_0, \qquad V_{i,j} = f^{k(i,j)}(G_{i,j}), \\
\hat{T}_{i,j} = \bigcup_{l=0}^{n(i,j)-1} f^{l}(V_{i,j}) ,
\end{gather*}
so that \(V_{i,j}\) is open and \(\hat{T}_{i,j}\) consists of the floors \(k(i,j), \ldots, i-1\) of \(T_{i,j}\) (see Figure~\ref{figuri}), and choose a point \(x_{i,j} \in V_{i,j} \cap \mathcal{N}_{N_0}\).
\begin{figure}[bt]
\centering
\definecolor{qqqqff}{rgb}{0,0,1}
\definecolor{zzttqq}{rgb}{0.6,0.2,0}
\begin{tikzpicture}[line cap=round,line join=round,>=triangle 45,x=1cm,y=1cm,scale=0.6]
\clip(1.8,-1.106043833688065) rectangle (12.2,11.3);
\fill[line width=2pt,color=zzttqq,fill=zzttqq,fill opacity=0.
10000000149011612] (4,10) -- (4,3) -- (6,3) -- (6,10) -- cycle;
\fill[line width=2pt,color=zzttqq,fill=zzttqq,fill opacity=0.
10000000149011612] (8,10) -- (8,0) -- (10,0) -- (10,10) -- cycle;
\draw [line width=2pt] (1.8,0)-- (4,0);
\draw [line width=2pt] (4,0)-- (10,0);
\draw [line width=2pt] (10,0)-- (12.2,0);
\draw [line width=2pt] (10,0)-- (10,10);
\draw [line width=2pt] (10,10)-- (4,10);
\draw [line width=2pt] (4,10)-- (4,0);
\draw [line width=2pt] (4,9)-- (2,9);
\draw [line width=2pt] (2,0)-- (2,9);
\draw (2.3,6) node[anchor=north west] {$T_{i-1}$};
\draw [line width=2pt] (10,10)-- (10,11);
\draw [line width=2pt] (12,11)-- (12,0);
\draw (10.2,7) node[anchor=north west] {$T_{i+1}$};
\draw [line width=2pt] (10,0)-- (11,0);
\draw [line width=2pt] (10,10)-- (4,10);
\draw [line width=2pt] (10,11)-- (12,11);
\draw [line width=1pt] (6,0)-- (6,10);
\draw [line width=1pt] (8,0)-- (8,10);
\draw (4.3,-0.25) node[anchor=north west] {$G_{i,1}$};
\draw (6.3,-0.25) node[anchor=north west] {$G_{i,2}$};
\draw (8.3,-0.25) node[anchor=north west] {$G_{i,3}$};
\draw (6.5,10.1) node[anchor=south west] {$T_i$};
\draw [line width=1pt,color=zzttqq] (4,3)-- (6,3);
\draw [line width=3pt,color=qqqqff] (4.2,3)-- (5.8,3);
\fill[line width=1pt,color=qqqqff,fill=qqqqff,fill opacity=0.10000000149011612] (4.2,10) -- (4.2,3) -- (5.8,3) -- (5.8,10) -- cycle;
\draw [line width=3pt,color=qqqqff] (8.3,0)-- (9.7,0);
\fill[line width=1pt,color=qqqqff,fill=qqqqff,fill opacity=0.10000000149011612] (8.3,10) -- (8.3,0) -- (9.7,0) -- (9.7,10) -- cycle;
\draw (4.3,6) node[anchor=north west] {$\hat{T}_{i,1}$};
\draw (8.3,6) node[anchor=north west] {$\hat{T}_{i,3}$};
\begin{scriptsize}
\draw[color=qqqqff] (5,2.5) node {$K_{i,1}$};
\draw[color=qqqqff] (9,0.35) node {$K_{i,3}$};
\end{scriptsize}
\end{tikzpicture}

\caption{A sketch of the construction within the tower \(T_i\).
 The base of the tower \(G_i\) is partitioned by \(G_{i,1},\ G_{i,2},\ G_{i,3}\).
 Only the compact sets \(K_{i,1}\) and \(K_{i,3}\) are defined.
}\label{figuri}
\end{figure}
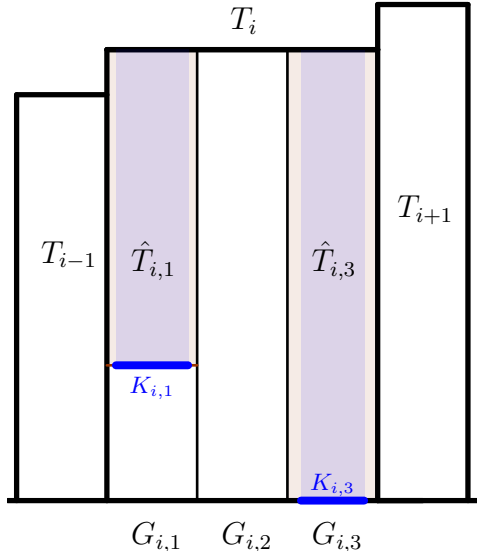
We check that \(x = x_{i,j}\), \(n = n(i,j)\) and \(V = V_{i,j}\) satisfy the hypotheses of Proposition~\ref{mainperturbativo}.
 First, \(x_{i,j} \in \mathcal{N}_{N_0} \subset \reg'\) and \(n(i,j) \geq N_0 \geq N_A(x_{i,j})\).
 Condition (i) holds because the sets \(f^{l}(V_{i,j})\), \(0 \leq l < n(i,j)\), are distinct floors of the tower \(T_{i,j}\).
 For condition (ii), let \(y \in f^{l}(V_{i,j})\) with \(0 \leq l < n(i,j)\); then \(y = f^{k(i,j)+l}(y')\) and \(x_{i,j} = f^{k(i,j)}(x')\) with \(y', x' \in G_{i,j}\), so \(d(y',x') < \eta\), and since \(k(i,j) + l \leq i - 1 < H\), \eqref{etaunif} gives
\[
d\big(f^{l}(y), f^{l}(x_{i,j})\big) = d\big(f^{k(i,j)+l}(y'), f^{k(i,j)+l}(x')\big) < \delta_{A} .
\]
By inner regularity of \(\mu\) and \(\mu(V_{i,j}) = \mu(G_{i,j})\), we choose a compact set \(K_{i,j} \subset V_{i,j}\) with
\begin{equation}\label{medidaK}
\mu(K_{i,j}) \geq (1 - \e_1)\,\mu(G_{i,j}) .
\end{equation}
Applying Proposition~\ref{mainperturbativo} with \(M\), \(\e\), \(\delta/2\), the point \(x_{i,j}\), \(n = n(i,j)\), \(V = V_{i,j}\) and \(K = K_{i,j}\), we obtain a cocycle \(B_{i,j} \in \text{\textup{C}}(X,\mathcal{G})\) satisfying properties 1 to 4 of that proposition, property 4 with \(\delta/2\).

Set
\[
B(y) = \begin{cases}
B_{i,j}(y) & \text{if } y \in \hat{T}_{i,j} \text{ for some } (i,j) \in \mathcal{J}, \\
A(y) & \text{otherwise.}
\end{cases}
\]
This is well defined because the sets \(\hat{T}_{i,j}\), \((i,j) \in \mathcal{J}\), are pairwise disjoint.
 By property 3 of Proposition~\ref{mainperturbativo}, \(B_{i,j} - A\) vanishes outside \(\bigcup_{l=0}^{n(i,j)-1} f^{l}(V_{i,j}) = \hat{T}_{i,j}\), hence
\[
B = A + \sum_{(i,j) \in \mathcal{J}} (B_{i,j} - A)
\]
is a finite sum of continuous maps; since at every point \(B\) coincides with \(A\) or with some \(B_{i,j}\), it takes values in \(\mathcal{G}\), so \(B \in \text{\textup{C}}(X,\mathcal{G})\), and \(\|A - B\| \leq \max_{(i,j) \in \mathcal{J}} \|A - B_{i,j}\| < \e\).
 This is item 1; in particular \eqref{cotac} applies to \(B\).
 Finally, for \((i,j) \in \mathcal{J}\) and \(y \in K_{i,j}\), the points \(y, f(y), \ldots, f^{n(i,j)-1}(y)\) lie in \(\hat{T}_{i,j}\), so \(B^{n(i,j)}(y) = B_{i,j}^{n(i,j)}(y)\) and property 4 gives
\begin{equation}\label{bloquebueno}
\log \|\wedge^p B^{n(i,j)}(y)\| \leq n(i,j)\Big(\Lambda' + \frac{\delta}{2}\Big) \qquad \text{for every } y \in K_{i,j} .
\end{equation}

\medskip
\noindent
\emph{Part 2: Estimation of \(\Lambda_p(B)\).}
\medskip

Let
\begin{gather*}
K = \bigcup_{(i,j) \in \mathcal{J}} K_{i,j}, \qquad
\hat{K} = \bigcup_{(i,j) \in \mathcal{J}}\ \bigcup_{l=0}^{n(i,j)-1} f^{l}(K_{i,j}), \\
n_{\max} = \max\{ n(i,j):\ (i,j) \in \mathcal{J} \} < H .
\end{gather*}
We claim that, up to a null set,
\begin{equation}\label{complemento}
X \setminus \hat{K} \subset \Big(\bigcup_{i \geq H} T_i\Big) \cup \mathcal{K} \cup \mathcal{T},
\end{equation}
where
\[
\mathcal{K} = {(\mathcal{N}_{N_0})}^{c} \cup \bigcup_{(i,j) \notin \mathcal{J}} T_{i,j}, \qquad
\mathcal{T} = \bigcup_{(i,j) \in \mathcal{J}}\ \bigcup_{l=0}^{n(i,j)-1} f^{l}\big(V_{i,j} \setminus K_{i,j}\big),
\]
the union defining \(\mathcal{K}\) running over the pairs \((i,j)\) with \(10N_0 \leq i < H\), \(1 \leq j \leq n(i)\), that are not in \(\mathcal{J}\).
 Indeed, discarding the null set \((X \setminus Q) \cup \bigcup_{i} \big( T_i \setminus \bigcup_{j} T_{i,j} \big)\), a point \(z \notin \bigcup_{i \geq H} T_i\) belongs to some \(T_{i,j}\) with \(i < H\).
 If \((i,j) \notin \mathcal{J}\), then \(z \in \mathcal{K}\).
 If \((i,j) \in \mathcal{J}\) and \(z\) lies in a floor \(l < k(i,j)\) of \(T_{i,j}\), then \(z \notin \mathcal{N}_{N_0}\) by the minimality of \(k(i,j)\), so \(z \in \mathcal{K}\).
 If \((i,j) \in \mathcal{J}\) and \(z\) lies in a floor \(l \geq k(i,j)\), then \(z \in \hat{T}_{i,j} = \bigcup_{l'} f^{l'}(V_{i,j})\), and \(z \notin \hat{K}\) forces \(z \in f^{l'}(V_{i,j} \setminus K_{i,j}) \subset \mathcal{T}\) for some \(l'\).

We now estimate the measures of the three sets in \eqref{complemento}.
 The first one has measure less than \(\e_1\) by \eqref{colatorres}.

For \(\mathcal{T}\): the sets \(f^{l}(V_{i,j} \setminus K_{i,j})\) appearing in its definition are pairwise disjoint, so by the invariance of \(\mu\), \eqref{medidaK} and \eqref{kac},
\begin{align*}
\mu(\mathcal{T}) &= \sum_{(i,j) \in \mathcal{J}} n(i,j)\,\mu\big(V_{i,j} \setminus K_{i,j}\big)
\leq \e_1 \sum_{(i,j) \in \mathcal{J}} i\,\mu(G_{i,j}) \\
&\leq \e_1 \sum_{i \geq 10N_0} i\,\mu(G_i) = \e_1 .
\end{align*}

For \(\mathcal{K}\): if \((i,j) \notin \mathcal{J}\), the floors \(0, 1, \ldots, i - N_0\) of \(T_{i,j}\), that is \(i - N_0 + 1\) of its \(i\) floors, all of the same measure, are contained in \({(\mathcal{N}_{N_0})}^{c}\).
 Since \(i \geq 10N_0\) we have \(\frac{i}{i - N_0 + 1} \leq \frac{10N_0}{9N_0 + 1} < \frac{10}{9}\), hence
\[
\mu(T_{i,j}) \leq \tfrac{10}{9}\,\mu\big(T_{i,j} \cap {(\mathcal{N}_{N_0})}^{c}\big) .
\]
Summing over \((i,j) \notin \mathcal{J}\) (the towers \(T_{i,j}\) are pairwise disjoint) and adding the part of \({(\mathcal{N}_{N_0})}^{c}\) outside these towers, \eqref{medidaN0} gives
\[
\mu(\mathcal{K}) \leq \tfrac{10}{9}\,\mu\big({(\mathcal{N}_{N_0})}^{c}\big) \leq \tfrac{10}{9} \cdot \tfrac{9}{10}\,\e_1 = \e_1 .
\]
Therefore
\begin{equation}\label{medidamala}
\mu(X \setminus \hat{K}) < 3\e_1 < \tfrac{1}{4} ,
\end{equation}
so \(\mu(\hat{K}) > 0\), and since \(\hat{K}\) is a finite union of images of \(K\) under iterates of \(f\), also \(\mu(K) > 0\).

By Birkhoff's ergodic theorem, Poincar\'e's recurrence theorem and the Furstenberg--Kesten theorem applied to the cocycle \(\wedge^p B\), for \(\mu\)-almost every \(y \in X\) the following hold:
\begin{enumerate}
\item[(a)] \(\lim_{N \to \infty} \frac{1}{N}\,\#\{0 \leq t < N:\ f^{t}(y) \notin \hat{K}\} = \mu(X \setminus \hat{K})\);
\item[(b)] if \(y \in K\), then \(f^{t}(y) \in K\) for infinitely many \(t \geq 1\);
\item[(c)] the limit \(\lim_{N \to \infty} \frac{1}{N} \log \|\wedge^p B^N(y)\|\) exists and equals \(\lambda_1^{\wedge p}(B) = \Lambda_p(B)\) (recall that \(\mu\) is ergodic).
\end{enumerate}
Since \(\mu(K) > 0\), we can fix \(y \in K\) satisfying (a), (b) and (c).

Let \(0 = t_1 < t_2 < \cdots\) be the successive times \(t \geq 0\) with \(f^{t}(y) \in K\) (an infinite sequence by (b)), and set \(y_i = f^{t_i}(y)\) and \(\tau_i = t_{i+1} - t_i\).
 For each \(i\) let \((i',j') \in \mathcal{J}\) be the pair with \(y_i \in K_{i',j'}\) and write \(n_i = n(i',j')\).
 We claim that \(\tau_i \geq n_i\).
 Indeed, \(y_i \in V_{i',j'}\) lies in the floor \(k(i',j')\) of \(T_{i',j'}\), so for \(1 \leq l \leq n_i - 1\) the point \(f^{l}(y_i)\) lies in the floor \(k(i',j') + l\) of the same tower.
 Since the floors of the subtowers are pairwise disjoint, \(K \cap T_{i',j'} = K_{i',j'}\) is contained in the floor \(k(i',j')\), and \(K \cap T_{i'',j''} = \emptyset\) whenever \((i'',j'') \neq (i',j')\), we get \(f^{l}(y_i) \notin K\) for \(1 \leq l \leq n_i - 1\), which proves the claim.

Fix \(N \geq 1\) and let \(k \geq 0\) be such that \(t_{k+1} \leq N < t_{k+2}\); write \(\hat{N} = N - t_{k+1} \in [0, \tau_{k+1})\).
 By the cocycle property and \(\tau_i \geq n_i\),
\[
B^{N}(y) = B^{\hat N}(y_{k+1})\, \prod_{i=k}^{1} \Big[ B^{\tau_i - n_i}\big(f^{n_i}(y_i)\big)\, B^{n_i}(y_i) \Big]
\]
(the product taken in decreasing order of \(i\)), and the same factorisation holds for \(\wedge^p B^{N}(y)\).
 Taking norms, and using \eqref{cotac} for the first factor and for the factors with exponent \(\tau_i - n_i\), and \eqref{bloquebueno} for the factors \(\wedge^p B^{n_i}(y_i)\), we obtain
\begin{equation}\label{cotaN}
\log \|\wedge^p B^{N}(y)\| \leq S_N \Big(\Lambda' + \frac{\delta}{2}\Big) + n_N \log c,
\end{equation}
where \(S_N = \sum_{i=1}^{k} n_i\) and \(n_N = N - S_N \geq 0\).
 Since \(S_N = N - n_N\) and \(|\Lambda'| \leq \log c\),
\begin{equation}\label{cotaNbis}
\frac{1}{N} \log \|\wedge^p B^{N}(y)\|
\leq \Lambda' + \frac{\delta}{2} + \frac{n_N}{N}\Big( \log c - \Lambda' - \frac{\delta}{2}\Big)
\leq \Lambda' + \frac{\delta}{2} + 2\log c\,\frac{n_N}{N} .
\end{equation}
(This step does not require \(\Lambda' + \delta/2 \geq 0\), which may fail when \(\mathcal{G} = \gld\).)

We claim that
\begin{equation}\label{sojourn}
n_N \leq \#\{0 \leq t < N:\ f^{t}(y) \notin \hat{K}\} + n_{\max} .
\end{equation}
Let \(0 \leq t < N\) with \(f^{t}(y) \in \hat{K}\), say \(f^{t}(y) \in f^{l}(K_{i,j})\) with \((i,j) \in \mathcal{J}\) and \(0 \leq l < n(i,j)\); the floor of \(T_{i,j}\) containing \(f^{t}(y)\) determines \(l\), and \(f^{t-l}(y) \in K_{i,j}\).
 If \(t - l < 0\), then \(y = f^{l-t}\big(f^{t-l}(y)\big)\) would lie in the floor \(k(i,j) + (l-t)\) of \(T_{i,j}\), with \(0 < l - t < n(i,j)\), while \(y \in K\) lies in the floor \(k(i',j')\) of some \(T_{i',j'}\); as the floors of the subtowers are pairwise disjoint, this forces \((i',j') = (i,j)\) and \(l - t = 0\), a contradiction.
 Hence \(t - l \geq 0\) is a visit time, \(t - l = t_{i_0}\) with \(i_0 \leq k+1\) (because \(t - l \leq t < t_{k+2}\)), and \(t = t_{i_0} + l\) with \(0 \leq l < n_{i_0}\).
 Therefore
\[
\#\{0 \leq t < N:\ f^{t}(y) \in \hat{K}\} \leq \sum_{i=1}^{k+1} n_i \leq S_N + n_{\max},
\]
which is \eqref{sojourn} because \(n_N = N - S_N\).
 By (a) and \eqref{medidamala},
\begin{equation}\label{limsojourn}
\limsup_{N \to \infty} \frac{n_N}{N} \leq \mu(X \setminus \hat{K}) < 3\e_1 .
\end{equation}

By (c), letting \(N \to \infty\) in \eqref{cotaNbis} and using \eqref{limsojourn} and \eqref{defeps1},
\begin{align*}
\Lambda_p(B) = \lim_{N \to \infty} \frac{1}{N} \log \|\wedge^p B^{N}(y)\|
&\leq \Lambda' + \frac{\delta}{2} + 2\log c \cdot \limsup_{N \to \infty}\frac{n_N}{N} \\
&< \Lambda' + \frac{\delta}{2} + 6\,\e_1 \log c = \Lambda' + \delta ,
\end{align*}
which is item 2.
 The theorem is proved.
\end{proof}

\section{Proof of the Main Theorem}

Throughout this section \(p \in \{1, \ldots, d-1\}\) is fixed, and for \(m \in \N\) we denote by
\[
\mathcal{D}_m \subset \text{\textup{C}}(X,\mathcal{G})
\]
the set of cocycles that admit an \(m\)-dominated splitting of index \(p\) over \(X\).
 We first record that non-domination at a fixed scale \(m\) is an open condition.

\begin{lema}\label{cerradodominado}
For every \(m \in \N\), the set \(\mathcal{D}_m\) is closed in \(\text{\textup{C}}(X,\mathcal{G})\).
\end{lema}

\begin{proof}
Let \(B_n \in \mathcal{D}_m\) with \(B_n \rightarrow B\) in \(\text{\textup{C}}(X,\mathcal{G})\), and let \(E_n \oplus F_n\) be the \(m\)-dominated splitting of index \(p\) of \(B_n\).
 Since \(B_n \to B\) uniformly, \(B\) takes values in \(\mathcal{G}\) and \(X\) is compact, \(B_n^{-1} \to B^{-1}\) uniformly as well, and the constant
\[
C = \sup_{n} \max\big\{ \|B_n\|,\ \|B_n^{-1}\| \big\}
\]
is finite.

\smallskip
Let \(x \in X\), and let \(u \in E_n(x)\), \(w \in F_n(x)\) be unit vectors.
 By \(m\)-domination,
\[
\|B_n^{m}(x)w\| \leq \big\|B_n^{m}(x)\restrict{F_n(x)}\big\| \leq \tfrac{1}{2}\,\text{\textup{\textbf{m}}}\big(B_n^{m}(x)\restrict{E_n(x)}\big) \leq \tfrac{1}{2}\|B_n^{m}(x)u\| ,
\]
hence
\begin{align*}
C^{m}\|u - w\| &\geq \|B_n^{m}(x)(u-w)\| \geq \|B_n^{m}(x)u\| - \|B_n^{m}(x)w\| \\
&\geq \tfrac{1}{2}\|B_n^{m}(x)u\| \geq \tfrac{1}{2}\,\text{\textup{\textbf{m}}}\big(B_n^{m}(x)\big) \geq \tfrac{1}{2}C^{-m} .
\end{align*}
Thus \(\|u - w\| \geq \frac{1}{2}C^{-2m}\) for all such \(u,w\), and since \(\|u - w\| = 2\sin\big(\measuredangle(u,w)/2\big)\), there is \(\theta_0 = \theta_0(m,C) > 0\), independent of \(n\) and \(x\), such that
\begin{equation}\label{angulouniforme}
\measuredangle\big(E_n(x), F_n(x)\big) \geq \theta_0 \qquad \text{for every } n \text{ and every } x \in X .
\end{equation}

\smallskip
Fix \(x \in X\).
 Since the Grassmannians of \(p\)- and \((d-p)\)-dimensional subspaces of \(\R^d\) are compact, a diagonal argument over \(k \in \Z\) gives a subsequence \((n_j)\) such that
\[
E_{n_j}(f^{k}(x)) \rightarrow E^{k} \quad \text{and} \quad F_{n_j}(f^{k}(x)) \rightarrow F^{k} \qquad \text{for every } k \in \Z .
\]
By \eqref{angulouniforme}, \(\measuredangle(E^{k},F^{k}) \geq \theta_0\), so \(\R^d = E^{k} \oplus F^{k}\) for every \(k\).
 Since \(B_{n_j}^{k}(x) \rightarrow B^{k}(x)\) and \(B_{n_j}^{k}(x)E_{n_j}(x) = E_{n_j}(f^{k}(x))\), we get \(B^{k}(x)E^{0} = E^{k}\), and likewise \(B^{k}(x)F^{0} = F^{k}\): the family \({\{E^{k} \oplus F^{k}\}}_{k \in \Z}\) is a \(B\)-invariant splitting of index \(p\) along the orbit of \(x\).
 Finally, the maps \((L,W) \mapsto \|L\restrict{W}\|\) and \((L,W) \mapsto \text{\textup{\textbf{m}}}(L\restrict{W})\) are continuous on \(\gld \times \{\text{subspaces}\}\), so the inequality
\[
\big\|B_{n_j}^{m}(f^{k}x)\restrict{F_{n_j}(f^{k}x)}\big\| \leq \tfrac{1}{2}\,\text{\textup{\textbf{m}}}\big(B_{n_j}^{m}(f^{k}x)\restrict{E_{n_j}(f^{k}x)}\big)
\]
passes to the limit \(j \to \infty\): the splitting \({\{E^{k} \oplus F^{k}\}}_{k}\) is \(m\)-dominated along the orbit of \(x\), with the same constant \(1/2\).

\smallskip
An orbit is an invariant set, so by the uniqueness of dominated splittings of a given index (Section~\ref{definitions}) any two invariant \(m\)-dominated splittings of index \(p\) along the orbit of \(x\) coincide.
 Hence the limits \(E^{0}\), \(F^{0}\) do not depend on the chosen subsequence: every subsequence of \((E_n(x),F_n(x))\) has a further subsequence converging to the same pair, so the whole sequence converges, \(E_n(x) \rightarrow E(x) := E^{0}\) and \(F_n(x) \rightarrow F(x) := F^{0}\).
 Doing this for every \(x \in X\) we obtain a \(B\)-invariant splitting \(\R^d = E(x) \oplus F(x)\) of index \(p\) which is \(m\)-dominated at every point of \(X\) (continuity then follows from the continuity property of Section~\ref{definitions}).
 Thus \(B \in \mathcal{D}_m\).
\end{proof}

We now translate Lemma~\ref{cerradodominado} into the measurable notion of non-domination used in Theorem~\ref{bajaexponente}.
 Recall from Section~\ref{definitions} that, when \(\lambda_p(A') > \lambda_{p+1}(A')\), the Oseledets splitting \(E' \oplus F'\) of index \(p\) of \(A'\) is defined on \(\reg(A')\), and that \(A'\) is said to be \emph{not \(m\)-dominated} if \(\mu\)-almost every point is eventually not \(m\)-dominated, that is, if the invariant set \(\mathcal{D}_{E'\oplus F'}(A',m)\) of points along whose orbit \(E' \oplus F'\) is \(m\)-dominated has measure zero.

\begin{lema}\label{nodominacionabierta}
Assume that \(A \in \text{\textup{C}}(X,\mathcal{G})\) admits no dominated splitting of index \(p\).
 Then for every \(m \in \N\) there exists \(r(m) > 0\) such that no cocycle in \(\mathcal{U}(A, r(m))\) admits an \(m\)-dominated splitting of index \(p\).
 Moreover, every \(A' \in \mathcal{U}(A,r(m))\) with \(\lambda_p(A') > \lambda_{p+1}(A')\) is not \(m\)-dominated.
\end{lema}

\begin{proof}
Since \(A \notin \mathcal{D}_m\) and \(\mathcal{D}_m\) is closed by Lemma~\ref{cerradodominado}, some ball \(\mathcal{U}(A,r(m))\) is disjoint from \(\mathcal{D}_m\); this is the first assertion.

For the second, let \(A' \in \mathcal{U}(A,r(m))\) with \(\lambda_p(A') > \lambda_{p+1}(A')\), and let \(D = \mathcal{D}_{E'\oplus F'}(A',m) \subset \reg(A')\).
 The set \(D\) is measurable (the Oseledets splitting is measurable, and \(D\) is a countable intersection of measurable conditions) and \(f\)-invariant, so \(\mu(D) \in \{0,1\}\) by ergodicity.
 Suppose \(\mu(D) = 1\).
 Then \(D\) is dense in \(X\), because \(\mu\) has full support, and the restriction of \(E' \oplus F'\) to the invariant set \(D\) is an \(m\)-dominated splitting of index \(p\).
 By the continuity property of Section~\ref{definitions} it extends continuously to a dominated splitting over \(\overline{D} = X\); since the extension at a point of \(X\) is the limit of the splitting at points of \(D\), the argument of limits along the orbits given in the proof of Lemma~\ref{cerradodominado} shows that the extended splitting is \(m\)-dominated with the same constant.
 Hence \(A' \in \mathcal{D}_m\), contradicting the first assertion.
 Therefore \(\mu(D) = 0\), that is, \(A'\) is not \(m\)-dominated.
\end{proof}

Note that Lemma~\ref{cerradodominado} does not say that the set of cocycles with a dominated splitting of index \(p\) (for \emph{some} \(m\)) is closed; a limit of \(m_n\)-dominated cocycles with \(m_n \to \infty\) need not be dominated.
 This is why the scale \(m\) must be fixed before choosing the perturbation \(A'\) in the proof below.

\medskip
For \(A \in \text{\textup{C}}(X,\mathcal{G})\) we define the \emph{limit inferior of \(\Lambda_p\) at \(A\)} as
\begin{equation}\label{defalpha}
\alpha(A) = \liminf_{B \rightarrow A} \Lambda_p(B) = \sup_{r > 0}\ \inf\big\{ \Lambda_p(B):\ \|B - A\| < r \big\} .
\end{equation}
The function \(r \mapsto \inf_{\mathcal{U}(A,r)} \Lambda_p\) is nonincreasing, so the supremum in \eqref{defalpha} is also the limit as \(r \downarrow 0\), and \(\alpha(A) \leq \Lambda_p(A)\).
 It is finite.
 Indeed, let \(\|A^{-1}\| = \max_{x \in X}\|A(x)^{-1}\|\) and \(r_0 = 1/(2\|A^{-1}\|)\).
 If \(\|B - A\| < r_0\), then for every \(x \in X\) we have \(\|A(x)^{-1}(A(x) - B(x))\| \leq 1/2\), so
\(B(x)^{-1} = \big(I - A(x)^{-1}(A(x) - B(x))\big)^{-1}A(x)^{-1}\) has norm at most \(2\|A^{-1}\|\), and since \(\lambda_d(B) \geq -\log\|B^{-1}\|\),
for every \(B \in \mathcal{U}(A,r_0)\) it holds that
\begin{equation}\label{cotainferior}
\Lambda_p(B) \geq p\,\lambda_d(B) \geq -p \log \|B^{-1}\| \geq -p\log\big(2\|A^{-1}\|\big) =: -K_A
\end{equation}
(for \(\mathcal{G} = \sld\) one simply has \(\Lambda_p \geq 0\)).
 We shall use \eqref{defalpha} only through the following two properties, valid for every \(\eta > 0\):
\begin{enumerate}
\item[(L1)] there exists \(r_1 = r_1(\eta) \in (0, r_0]\) such that \(\Lambda_p(B) > \alpha(A) - \eta\) for every \(B \in \mathcal{U}(A,r_1)\);
\item[(L2)] for every \(r > 0\) there exists \(A' \in \mathcal{U}(A,r)\) with \(\Lambda_p(A') < \alpha(A) + \eta\).
\end{enumerate}
Property (L1) holds because \(\inf_{\mathcal{U}(A,r)} \Lambda_p\) increases to the finite number \(\alpha(A)\) as \(r \downarrow 0\), and (L2) because \(\inf_{\mathcal{U}(A,r)} \Lambda_p \leq \alpha(A)\) for every \(r > 0\).

The idea of the proof is now simple to state: a cocycle \(A'\) near \(A\) with \(\Lambda_p(A')\) close to \(\alpha(A)\) is an almost minimiser of \(\Lambda_p\); Theorem~\ref{bajaexponente} lowers \(\Lambda_p\) by almost half of the gap \(\beta\) without leaving the neighbourhood where (L1) holds, and this is incompatible with the almost minimality.

\begin{proof}[Proof of the Main Theorem]
Suppose, by contradiction, that \(A\) has the uniform \(p\)-gap property but admits no dominated splitting of index \(p\).
 We proceed in seven steps; the parameters are chosen in the order in which they appear.

\smallskip
By hypothesis there are \(\e_0 \in (0,1)\) and \(\beta > 0\) such that
\begin{equation}\label{gap}
\lambda_p(B) - \lambda_{p+1}(B) \geq \beta \qquad \text{for every } B \in \mathcal{U}(A,\e_0) .
\end{equation}
Set
\[
\eta = \frac{\beta}{8} \qquad \text{and} \qquad M = 2 + \max\big\{ \|A\|,\ 2\|A^{-1}\| \big\} .
\]

\smallskip
Let \(\rho = \min\{\e_0, r_0\} < 1\), with \(r_0 = 1/(2\|A^{-1}\|)\) as in \eqref{cotainferior}.
 Every \(B \in \mathcal{U}(A,\rho)\) satisfies \eqref{gap}, in particular \(\lambda_p(B) > \lambda_{p+1}(B)\) and the Oseledets splitting of index \(p\) of \(B\) is defined; moreover \(\|B\| < \|A\| + 1 < M\), \(\|B^{-1}\| \leq 2\|A^{-1}\| < M\) and \(\Lambda_p(B) \geq -K_A\).

\smallskip
Let \(\alpha = \alpha(A)\) be given by \eqref{defalpha}; it is finite by \eqref{cotainferior}.
 By (L1) there is \(r_1 \in (0,\rho]\) such that
\begin{equation}\label{cotaL1}
\Lambda_p(B) > \alpha - \eta \qquad \text{for every } B \in \mathcal{U}(A,r_1) .
\end{equation}

\smallskip
Set
\[
\e_T = \frac{r_1}{2} \in (0,1), \qquad \e_T' = \frac{\e_T}{8M^2} \qquad \text{and} \qquad m_1 = m_0\big(M, \e_T'\big),
\]
with \(m_0\) given by Proposition~\ref{lemaperturbativo}; that is, \(\e_T' = (\e_T)_M\) in the notation of \eqref{defepsM}, because \(\e_T < 1\), and \(m_1\) is the threshold required by Theorem~\ref{bajaexponente} for the parameters \(M\) and \(\e_T\).
 Note that \(\e_T\) and \(m_1\) depend only on \(A\) and \(\beta\); in particular they are fixed before \(A'\) is chosen.

\smallskip
By Lemma~\ref{nodominacionabierta} applied with \(m = m_1\), there is \(r_2 = r(m_1) > 0\) such that every \(A' \in \mathcal{U}(A,r_2)\) with \(\lambda_p(A') > \lambda_{p+1}(A')\) is not \(m_1\)-dominated.

\smallskip
By (L2) with \(r = \min\{\e_T, r_2\}\), there exists \(A' \in \text{\textup{C}}(X,\mathcal{G})\) such that
\[
\|A' - A\| < \min\{\e_T, r_2\} \qquad \text{and} \qquad \Lambda_p(A') < \alpha + \eta .
\]
Then:
\begin{enumerate}
\item[(P1)] \(A' \in \mathcal{U}(A,\rho)\), so \(\|A'\| < M\) and \(\|{A'}^{-1}\| < M\);
\item[(P2)] \(\lambda_p(A') - \lambda_{p+1}(A') \geq \beta\) by \eqref{gap}; in particular the Oseledets splitting of index \(p\) of \(A'\) is defined;
\item[(P3)] \(A'\) is not \(m_1\)-dominated, with \(m_1 = m_0(M,\e_T')\);
\item[(P4)] \(\Lambda_p(A') < \alpha + \eta\).
\end{enumerate}
By (P1), (P2) and (P3), the cocycle \(A'\) satisfies the hypotheses of Theorem~\ref{bajaexponente} with the parameters \(M\), \(\e = \e_T\) and \(\delta = \eta\).

\smallskip
Theorem~\ref{bajaexponente} provides \(B \in \text{\textup{C}}(X,\mathcal{G})\) with \(\|B - A'\| < \e_T\) and
\begin{align*}
\Lambda_p(B) &< \frac{\Lambda_{p-1}(A') + \Lambda_{p+1}(A')}{2} + \eta
 = \Lambda_p(A') - \frac{\lambda_p(A') - \lambda_{p+1}(A')}{2} + \eta \\
 &\leq \Lambda_p(A') - \frac{\beta}{2} + \eta ,
\end{align*}
where the equality follows from \(\Lambda_{p-1} = \Lambda_p - \lambda_p\) and \(\Lambda_{p+1} = \Lambda_p + \lambda_{p+1}\), and the last inequality from (P2).

\smallskip
On the one hand,
\(\|B - A\| \leq \|B - A'\| + \|A' - A\| < 2\e_T = r_1\),
so \eqref{cotaL1} gives \(\Lambda_p(B) > \alpha - \eta\).
 On the other hand, (P4) implies,
\[
\Lambda_p(B) < \alpha + \eta - \frac{\beta}{2} + \eta = \alpha - \frac{\beta}{4} = \alpha - 2\eta < \alpha - \eta .
\]
This contradiction shows that \(A\) admits a dominated splitting of index \(p\), which is the Main Theorem.
\end{proof}

We now prove Corollary~\ref{mncoro}.

\begin{proof}[Proof of Corollary~\ref{mncoro}]
Let \(A \in \text{\textup{C}}(X,\text{\textup{SL}}(3,\mathbb{R}))\) be a point of continuity of the second Lyapunov exponent and assume that \(\lambda_2(A) \neq 0\).
We consider the case \(\lambda_2(A) < 0\); the case \(\lambda_2(A) > 0\) is analogous, exchanging the roles of the indices \(1\) and \(2\) and of the exponents \(\lambda_1\) and \(\lambda_3\).
Recall that \(\lambda_1(B) + \lambda_2(B) + \lambda_3(B) = 0\) for every \(B \in \text{\textup{C}}(X,\text{\textup{SL}}(3,\mathbb{R}))\).
In particular, \(\lambda_1(A) > 0 > \lambda_2(A)\).

\smallskip
\noindent\emph{Step 1: \(A\) admits a dominated splitting of index \(1\).}
Suppose that this is not the case.
By the Main Theorem, the Lyapunov spectrum of \(A\) does not have the uniform \(1\)-gap property: for every neighbourhood \(\mathcal{U}\) of \(A\) and every \(\beta > 0\) there is \(B \in \mathcal{U}\) such that \(\lambda_1(B) - \lambda_2(B) < \beta\).
Take \(\beta = -\lambda_2(A)/2\) and, using the continuity of \(\lambda_2\) at \(A\), a neighbourhood \(\mathcal{U}\) of \(A\) such that \(\lambda_2(B) < -\beta\) for every \(B \in \mathcal{U}\).
Then there is \(B \in \mathcal{U}\) with \(\lambda_1(B) < \lambda_2(B) + \beta < 0\), and so \(\lambda_1(B) + \lambda_2(B) + \lambda_3(B) \leq 3\lambda_1(B) < 0\), a contradiction.

\smallskip
\noindent\emph{Step 2: \(A\) is a point of continuity of \(\lambda_1\) and \(\lambda_3\).}
The dominated splitting of index \(1\) given by Step~1 coincides almost everywhere with \(E_1 \oplus (E_2 \oplus E_3)\), where \(E_1\) is the Oseledets subspace of \(\lambda_1(A)\); we still denote it by \(E_1 \oplus F\).
Dominated splittings persist under \(C^0\)-small perturbations of the cocycle and their bundles vary continuously with the cocycle (see~\cite[Appendix~B.1]{BonattiDiazViana}): every \(B\) close to \(A\) admits a dominated splitting \(E_1^B \oplus F^B\) of index \(1\), and \(E_1^B(x)\) converges to \(E_1(x)\), uniformly in \(x\), as \(B\) converges to \(A\).
Since \(E_1^B\) is one-dimensional and carries the largest Lyapunov exponent of \(B\), Birkhoff's ergodic theorem gives
\[
\lambda_1(B) = \int_X \log \| B(x)\restrict{E_1^B(x)} \| \, d\mu(x),
\]
and the right hand side is continuous at \(B = A\).
Hence \(\lambda_1\) is continuous at \(A\), and so is \(\lambda_3 = -\lambda_1 - \lambda_2\).

\smallskip
\noindent\emph{Step 3:}
If \(\lambda_2(A) = \lambda_3(A)\), the Oseledets splitting of \(A\) has two subspaces and coincides almost everywhere with \(E_1 \oplus F\), which is dominated.
If \(\lambda_2(A) > \lambda_3(A)\), then \(A\) is a point of continuity of the second and the third Lyapunov exponents, and Corollary~\ref{corocontinuidad} with \(p = 2\) shows that \(A\) admits a dominated splitting of index \(2\).
Together with Step~1, this means that the Oseledets splitting \(E_1 \oplus E_2 \oplus E_3\) of \(A\) is dominated.
\end{proof}

\section{Some questions}

In this section we pose some questions related to the Main Theorem.
The first one concerns its consequences.
Suppose that a cocycle \(A \in \text{\textup{C}}(X,\mathcal{G})\) satisfies \(\lambda_p(A) > \lambda_{p+1}(A)\) and that its Oseledets splitting \(E \oplus F\) of index \(p\) is not dominated.
By the Main Theorem, for every \(C^0\)-neighbourhood \(\mathcal{U}\) of \(A\) and every \(\e > 0\) there is a cocycle \(B \in \mathcal{U}\) such that \(0\leq \lambda_{p}(B)-\lambda_{p+1}(B)<\e\); that is, the two exponents can be made arbitrarily close.

\begin{quesnum}
Does there always exist a cocycle \(B \in \mathcal{U}\) such that \(\lambda_p(B) = \lambda_{p+1}(B)\)?
\end{quesnum}

Recall the \(C^0\)-generic dichotomy of~\cite{BochiViana_2005}: for generic cocycles, either the Oseledets splitting is dominated or all the Lyapunov exponents are equal.
However, this fact alone does not seem to be enough to answer the question.

The next question is about removing the assumption that \(\mu\) is ergodic.

\begin{quesnum}
Is the Main Theorem valid if the measure \(\mu\) is not ergodic?
\end{quesnum}
This question requires adapting the definition of the uniform \(p\)-gap property, since the Lyapunov exponents are no longer constant.
A natural choice is to require the gap on average, that is, for the integrated Lyapunov exponents:
there exist a \(C^0\)-neighbourhood \(\mathcal{U}\) of \(A\) and \(\beta >0\) such that, for every \(B \in \mathcal{U}\),
\[
\int_X \left[\lambda_p(B,x) - \lambda_{p+1}(B,x)\right]d\mu(x) > \beta.
\]

However, with this definition the question seems to have a negative answer\footnote{Thanks to Mauricio Poletti for pointing this out and providing the sketch of the argument that is reproduced here.
}.
We sketch the argument. Let \(X\) and \(f\) be as in the Main Theorem.
Suppose that there are two disjoint compact invariant sets \(\Lambda_1\) and \(\Lambda_2\) supporting ergodic Borel probability measures \(\mu_1\) and \(\mu_2\), respectively.
Consider a cocycle \(A \in \text{\textup{C}}(X,\gld)\) such that \(A\restrict{\Lambda_1}\) has the uniform \(p\)-gap property, while \(A\restrict{\Lambda_2}\) satisfies \(\lambda_p(A\restrict{\Lambda_2}) = \lambda_{p+1}(A\restrict{\Lambda_2})\) and admits no dominated splitting of index \(p\).

Consider the measure \(\mu = \frac{1}{2}\mu_1 + \frac{1}{2}\mu_2\).
It is a non-ergodic \(f\)-invariant Borel probability measure.
The cocycle \(A\) has the uniform \(p\)-gap property in the averaged sense above, since \(\lambda_p - \lambda_{p+1} \geq 0\) on \(\Lambda_2\) for every cocycle.
However, \(A\) does not admit a dominated splitting of index \(p\) over \(\supp(\mu)\).

A reformulation of Question 2 that may be of interest is the following.

\begin{ques2rev}
Under the hypotheses of the Main Theorem, except that \(\mu\) is not ergodic, does there exist an invariant set \(\Lambda \subset X\) with \(\mu(\Lambda)>0\) over which \(A\) admits a dominated splitting of index \(p\)?
\end{ques2rev}

Another question concerns cocycles with values in other matrix groups.

\begin{quesnum}
Is the Main Theorem valid for cocycles with values in other subgroups of \(\gld\)?
\end{quesnum}

In the proof of Proposition~\ref{lemaperturbativo} we used that the cocycle takes values in \(\gld\) or \(\sld\); more precisely, that it can be perturbed by composing with rotations and with hyperbolic matrices of determinant one.
All the subsequent constructions rely on this.
Thus the question amounts to establishing the conclusion of Proposition~\ref{lemaperturbativo} for other matrix groups.

In~\cite{BochiViana_2005} the authors give conditions for the existence of space exchanging sequences for a fixed cocycle.
There the cocycle may take values in any submanifold of \(\gld\) satisfying an accessibility condition, which holds, for instance, for every subgroup of \(\gld\) acting transitively on \(\P\R^d\).
One way of approaching our question is to check whether the conclusion of Proposition~\ref{lemaperturbativo} holds for cocycles with values in such submanifolds, with a threshold \(m_0\) depending only on \(M\) and \(\e\).

As a final question, let \(N\) be a compact Riemannian manifold, let \(\mu\) be the volume measure and let \(f \in \Diff[\mu]{N}\) be a volume preserving diffeomorphism.
As in Question~2, we say that its Lyapunov spectrum has the \emph{uniform \(p\)-gap property} if there exist a \(C^1\)-neighbourhood \(\mathcal{U}_f\) of \(f\) and \(\beta>0\) such that, for every \(g \in \mathcal{U}_f\),
\[
\int_N [\lambda_p(g,x) - \lambda_{p+1}(g,x)]\,d\mu(x) > \beta.
\]
The analogue of the Main Theorem in this context is the following question.

\begin{quesnum}
Does the uniform \(p\)-gap property imply that \(f\) admits a dominated splitting of index \(p\)?
\end{quesnum}
One may also ask the same question restricted to ergodic diffeomorphisms.

\subsection*{Acknowledgements} The author would like to thank Radu Saghin and Francisco Valenzuela-Henríquez for their guidance and encouragement during his PhD at PUCV, where this work originated, and for their numerous suggestions to improve it.
 The author is also grateful to Lucas Backes for pointing out mistakes in Proposition~\ref{lemaperturbativo}.
 The author also thanks Jairo Bochi, Mauricio Poletti and Rafael Potrie for their comments, which improved the readability of the paper, and for their suggestions for further work.

\bibliographystyle{amsplain}
\bibliography{bibliografia}

\end{document}